\documentclass{article}
\usepackage{amsmath}
\usepackage{amsthm}
\usepackage{amssymb}
\usepackage{epstopdf}
\usepackage{graphicx}
\newtheorem{theorem}{Theorem}
\newtheorem{corollary}{Corollary}
\newtheorem{lemma}{Lemma}

\begin{document}
\title%Optional Short Title
{Arguments Related to the Riemann Hypothesis: New Methods and Results}
\author{R.C. McPhedran,
School of Physics,\\
University of Sydney 2006, Australia}
\maketitle
%\subject{Analytic number theory, Riemann hypothesis}
%\keywords{Riemann hypothesis, number theory, analytic functions}
%\corres{R. C. McPhedran}
%\email{ross.mcphedran@sydney.edu.au}
\begin{abstract}
Four propositions are considered concerning  the relationship between the zeros of two combinations of the Riemann zeta function and the function itself. The first is the Riemann hypothesis, while the second relates to the zeros of a derivative function. It is proved
that these are equivalent, and that, if the Riemann hypothesis holds, then all zeros of the zeta function on the critical line are simple.
The Riemann hypothesis is then shown to imply the third proposition holds, this being a new necessary condition for the Riemann hypothesis. The third proposition is shown to be equivalent to the fourth, and either is shown to yield the result that the distribution of zeros on the critical line of $\zeta (s)$ is that given by the Riemann hypothesis.The results given are obtained from a combination of analytic arguments, experimental mathematical techniques and graphical reasoning.
\end{abstract}

\section{Introduction}
The Riemann hypothesis is regarded as one of the most important and difficult unsolved problems in mathematics. The hypothesis is that all complex  zeros of the Riemann zeta function $\zeta (s)$ where $\Re (s)=\sigma$, $\Im (s)=t$ lie on $\sigma=1/2$ \cite{edwards, titchmarsh}. In the nearly  hundred and fifty years since Riemann's paper, an extensive literature has developed around the  hypothesis and other  properties of $\zeta (s)$, including many deep analytic results and extensive numerical investigations. However, the hypothesis remains unproved.

Despite the difficulty of this topic, new methods such as those of experimental mathematics \cite{borbail} have arisen, which may permit
the derivation of new results, when combined with a body of established results from the literature. Such a combination is presented in this paper.

The approach followed here relies on certain basic properties of the zeta function, which we will now discuss, based on the texts \cite{edwards, titchmarsh}.  The zeta function has the series
\begin{equation}
\zeta (s)=\sum_{n=1}^\infty \frac{1}{n^s},
\label{intro1}
\end{equation}
which converges absolutely in  $\sigma>1$. It may be continued analytically  into $\sigma<0$ using the functional equation
\begin{equation}
\xi_1 (s)=\frac{\Gamma (s/2) \zeta (s)}{\pi^{s/2}}=\xi_1(1-s),
\label{intro2}
\end{equation}
with an integral form defining it in the {\em critical strip} $0\leq\sigma\leq1$. A variant of the function $\xi_1(s)$ having the same symmetric property under $s\rightarrow 1-s$ is:
\begin{equation}
\xi (s)=\frac{1}{2} s(s-1) \xi_1(s)=\xi (1-s).
\label{intro2a}
\end{equation}
 Riemann was able to prove using the Euler product relation for $\zeta (s)$ that all its zeros lie in the critical strip, and asserted that the number of zeros whose imaginary parts lie between 0 and $T$ is approximately
\begin{equation}
\frac{T}{2\pi} \log \frac{T}{2\pi} -\frac{T}{2\pi} ,
\label{intro3}
\end{equation}
the relative error in the approximation being of order $1/T$. (This was in fact proved by von Mangoldt in 1905 \cite{edwards}.)

Knowing the formula (\ref{intro3}) for the distribution of zeros in the critical strip, the crucial question becomes that of determining the proportion of these lying actually on the critical line. There have been notable successes in proving results close, in some sense, to the Riemann hypothesis that the distribution formula for zeros on the critical line is also given by (\ref{intro3}). For example, Bohr and Landau \cite{edwards, bland} proved that the number of complex roots $t$ of $\xi (1/2+i t)=0$ with $0\leq \Re (t)\leq T$ and $-\varepsilon\leq \Im(t)\leq \varepsilon$ is equal to (\ref{intro3}), for any $\varepsilon\geq 0$, with a relative error which approaches zero as $T\rightarrow \infty$. As another example, Bui, Conrey and Young \cite{bcy} have proved that over 41\% of the zeros of $\zeta (s)$ must lie on the critical line. Extensive numerical studies have been carried out on the zeros  of $\zeta (s)$; it is known that  all such zeros are simple and lie on the critical line, as far as the first $O(10^9)$ zeros are concerned \cite{titchmarsh}, p.391.

Of direct importance to the remainder of this paper are the results concerning antisymmetric/symmetric combinations of two terms involving $\xi (s)$ or
$\xi_1 (s)$ with the arguments differing by one. These are of interest because they have been shown to have the properties which one would like to prove hold true for $\zeta (s)$: all their complex zeros lie on the critical line, and are simple. The functions were studied by Taylor \cite{prt}, Lagarias and Suzuki \cite{lagandsuz} and Ki\cite{ki}. The  fact that all the non-trivial zeros  of the antisymmetric  combination lie on the critical line was first established by P.R. Taylor, and published  
posthumously. (Taylor was in fact killed on active duty with the RAF in North Africa during World War II; the paper was compiled from his notes by Mr. J.E. Rees, while the argument was revised and completed by Professor Titchmarsh.) Lagarias and Suzuki considered the symmetric combination, and showed that all its complex zeros lie on the critical line, while Ki proved that all the complex zeros were simple. A further useful property is that the complex zeros of the symmetric and antisymmetric combinations strictly alternate on the critical line, and have the same distribution function of zeros. The common distribution function is indeed that corresponding to any prescribed argument value of $\xi_1$ on the line $\sigma=1$.

In the remainder of this paper, we build on the work of Taylor \cite{prt}, Lagarias and Suzuki \cite{lagandsuz} and Ki\cite{ki}. 
Section 2 contains a discussion of an important figure, which encapsulates the key results of this paper. It is followed by a discussion of
key theorems and definitions from  the work of Taylor \cite{prt}, Lagarias and Suzuki \cite{lagandsuz} and Ki\cite{ki}. In Section 4, two functions denoted ${\cal U}(s)$ and ${\cal V}(s)$ are studied.
The first is formed from the ratio of $\xi_1( 2 s-1)$ and $\xi_1(2 s)$, while the second is constructed from the ratio of the sum and difference of the two preceding functions. (The factor of two multiplying $s$ in the arguments of these functions is inserted for generality; it has the result that the distribution function of zeros then becomes the same as that of certain double sums: see Chapter 3 of Borwein {\em et al} \cite{lsbook} and further comments in Section 3.) The function ${\cal V}(s)$ has all its complex zeros and poles placed alternately on the critical line:
we can use this to reason {\em from the inside out}: by this we mean that the central properties are exploited as we move away from
$\sigma=1/2$.  The function ${\cal U}(s)$ has its zeros in $1>\sigma>1/2$ and its poles in $0<\sigma<1/2$: the distribution function of each in the critical strip is given  by (\ref{intro3}) (if $T$ is replaced by $2 T$ to take into account the replacement of $s$ by $2 s$).
We may refer to ${\cal U}(s)$ as enabling reasoning {\em from the outside in}. 

Section 4 introduces four propositions, which give structure to the new results presented in it and Section 5. The first proposition is the Riemann hypothesis, as applied to $\zeta(2 s-1)$ and $\zeta(2 s)$. The second is in the style of arguments concerning the location of zeros of $\zeta'(s)$ first put forward by Speiser \cite{speiser}. The third and fourth are new, and refer to the modulus of  ${\cal V}(s)$ at points where its derivative is zero, and to the connection between zeros of ${\cal V}(s)$ and points on the critical line where its modulus is unity.

In Section 4, the first important result established is that, if the Riemann hypothesis holds then the real part of the derivative of
$\log {\cal U}(s)$ does not change sign either between the lines $\sigma=3/4$ and $\sigma =1/4$ or on them (guaranteeing Proposition 2 holds).  Consequently, the Riemann hypothesis implies that all zeros of $\zeta (s)$ on the critical line are simple.  It is next shown
that Proposition 2 implies Proposition 1, so the two are equivalent. Section 5 is concerned with the new propositions. It is shown that
if Propositions 1 and 2 hold, this  implies Proposition 3 holds. Also, Propositions 3 and 4 are shown to be equivalent. Going in the reverse direction, it is shown that if Proposition 3 holds, then the distribution function of zeros of $\zeta (2s -1/2)$ on the critical line is that given by the Riemann hypothesis. Section 5 ends with a discussion of the topology of the regions defined by the argument of ${\cal U}(s)$.

The key differences between this and the previous version of this paper relate to the results given here in Sections 4.1 and 4.2. These algebraic arguments have been placed earlier in the current version, as they may be more readily appreciated by workers in the field than the graphical arguments of the previous version. The more graphical arguments are mostly given in Section 5, although in Section 4.2 a graphical example is given illustrating an artificial counterexample to the Riemann hypothesis, and the resulting differences in functional structures.

Key points in the behaviour of the analytic functions studied are graphically illustrated in the paper. Supplementary material gives the zeros of the functions ${\cal T}_+(s)$ and  ${\cal T}_-(s)$ on the critical line up to $t=1000$. The definitions of all functions studied are easily implemented in sophisticated computer packages like Mathematica and Maple.
\section{A Revealing Figure}
The topic concerning us in this paper is, in essence, centred around the properties of two functions. That denoted ${\cal V}(s)$ has its fundamental properties well determined. It is a ratio of the symmetric and antisymmetric combinations of  functions $\xi (2 s-1)$ and $\xi (2 s)$, all its zeros and poles lie on the critical line $\sigma=1/2$, and their distribution function is established \cite{prt}-\cite{ki}. For that denoted ${\cal U}(s)$, 
where
\begin{equation}
{\cal U}(s)=\frac{\xi_1(2s -1)}{\xi_1(2 s)},
\label{eq-srh2}
\end{equation}
the location of its zeros and poles is the subject of the important but unproved Riemann hypothesis, it is undetermined whether they are of first or higher order, and the distribution function of its zeros {\em on the critical line} is keenly sought. It is known, however, that the distribution function of its zeros {\em within the critical strip} $0<\sigma<1$ is exactly that pertaining to the zeros or poles of ${\cal V}(s)$
on the critical line. We would like to link as tightly as possible the properties of the ``target function"  ${\cal U}(s)$ to those of the ``reference function''  ${\cal V}(s)$.

In order to do this, we consider the simple relationship linking these two functions:
\begin{equation}
{\cal V}(s)=\frac{1+{\cal U}(s)}{1-{\cal U}(s)}.
\label{vfromu}
\end{equation}
It is readily seen that this transformation has two fixed points ${\cal V}(s)={\cal U}(s)=i$ and ${\cal V}(s)={\cal U}(s)=-i$. These fixed points are made the zeros and poles of a  function introduced previously \cite{mcp13}:
\begin{equation}
{\cal W}(s)=\frac{{\cal V}(s)-i}{{\cal V}(s)+i}=i \left( \frac{{\cal U}(s)-i}{{\cal U}(s)+i}\right).
\label{wdef}
\end{equation}
Figure \ref{fig-revealing} shows contours $|{\cal W}(s)|=1$, and compares them with contours $|{\cal V}(s)|=1$.

 \begin{figure}[tbh]
\includegraphics[width=10cm]{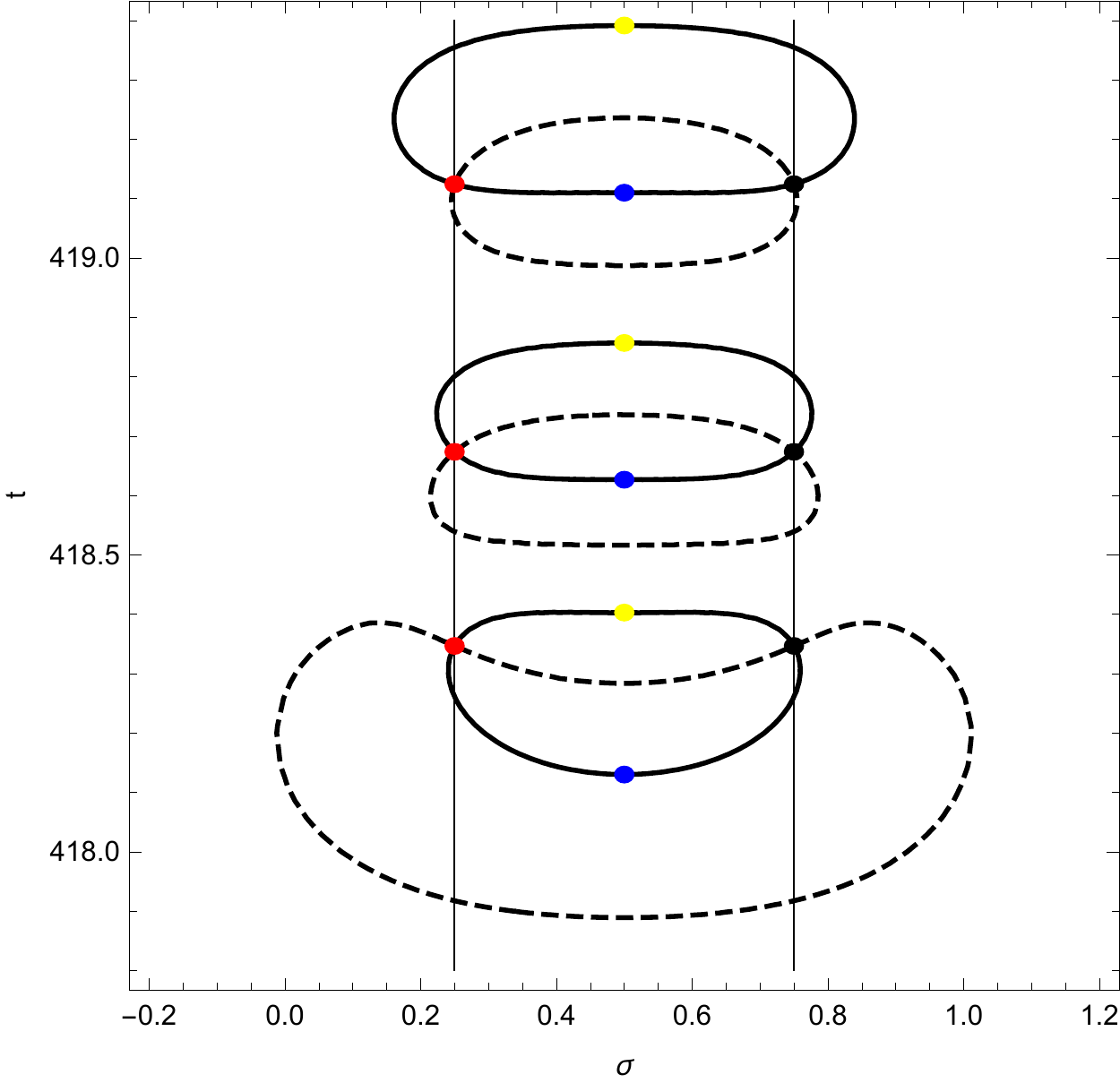}
\caption{The black full lines denote contours $|{\cal W}(s)|=1$, while the  black dashed lines denote contours $|{\cal V}(s)|=1$. 
Blue dots denote zeros of ${\cal V}(s)$, while yellow dots denote poles. Black dots denote zeros of ${\cal U}(s)$, while red dots denote poles. The Riemann hypothesis locates zeros of ${\cal U}(s)$ on $\sigma=3/4$, and poles on $\sigma=1/4$.}
\label{fig-revealing}
\end{figure}

What is evident in Fig. \ref{fig-revealing} is that zeros and poles of ${\cal U}(s)$ sit between zeros and poles of ${\cal V}(s)$ on contours
$|{\cal W}(s)|=1$, with the former lying where $|{\cal W}(s)|=1$ intersects $|{\cal V}(s)|=1$. These properties are important, because they establish a one-to-one correspondence between zeros of ${\cal V}(s)$ or ${\cal W}(s)$ and those of ${\cal U}(s)$. As remarked, we know that the density function of the first is that corresponding to that of  the zeros of  ${\cal U}(s)$ everywhere in $0<\sigma<1$ . However,
zeros of ${\cal U}(s)$ off the  line $\sigma=3/4$ occur not singly, but in pairs, so that such pairs must be sufficiently rare  to not disturb the
distribution function. Hence, we can deduce from Fig. \ref{fig-revealing}, {\em if the geometry shown holds for all zeros} of ${\cal W}(s)$, then almost all zeros of ${\cal U}(s)$ lie  on the  line $\sigma =3/4$. We also know that zeros of ${\cal U}(s)$ lying on a contour $|{\cal W}(s)|=1$ which has a simple zero of ${\cal W}(s)$ at its centre must also be simple.  Hence, we further deduce that almost all zeros of ${\cal U}(s)$ are simple.

It is one purpose of the remainder of this paper to unite known results  and new results in order to show that Fig. \ref{fig-revealing} represents a universal behaviour.

\section{Properties of Combinations of Zeta Functions}
In this section, we review the properties of combinations of $\xi$ functions and their zeros, which will be needed for the arguments presented in Section 3.

\subsection{The Results of Taylor}
The posthumous paper of P.R. Taylor contains results on five related topics, only the first of which is of interest here. The following result is proved:
\begin{theorem}
(Taylor)  The function 
$f(s)=\xi_1(s+1/2)-\xi_1(s-1/2)$, odd under $s\rightarrow 1-s$,  has all its complex zeros on $\sigma=1/2$.
\label{prtthm}
\end{theorem}

Taylor's proof examines using the argument principle the difference between the number of zeros in the critical strip and on the critical line with $0<t<T$, and bounding the difference in such a way as to give the result.

\subsection{The Results of Lagarias and Suzuki}
The substantial paper by Lagarias and Suzuki \cite{lagandsuz} was written at around the same time as that of Ki \cite{ki}. Although the two papers overlap in some results, there are results in both we will need for the arguments of Section 3.

The main theorems proved by Lagarias and Suzuki are now given, using our notations.

\begin{theorem}(Lagarias and Suzuki )

The meromorphic function, even under $s\rightarrow 1-s$,
\begin{equation}
\xi_1(2 s) \frac{1}{s-1}-\xi_1(2 s-1) \frac{1}{s}
\label{lseq1}
\end{equation}
has all its zeros on the critical line.
\label{lsthm1}
\end{theorem}
The next result quoted is for a more general function, again even under $s\rightarrow 1-s$:
\begin{theorem}(Lagarias and Suzuki )

For each fixed $T\geq 1$, the meromorphic function
\begin{equation}
I(T,s)=-\xi_1(2 s) \frac{T^{s-1}}{s-1}+\xi_1(2 -2 s) \frac{T^{-s}}{s}
\label{lseq2}
\end{equation}
has all its zeros on the critical line.
\label{lsthm2}
\end{theorem}

The third example is motivated by an Eisenstein series result, commonly viewed as due to Chowla and Selberg \cite{chowsel}, but in its essence predated by a paper of H. Kober \cite{kober}.  Kober's equation (5a) is:
\begin{eqnarray}
&& (a c-b^2)^{s/2-1/4} a^{1/2}\frac{\Gamma (s)}{8\pi^s} \sum_{(m_1,m_2)\neq (0,0)} \frac{1}{(a m_1^2 +2 b m_1 m_2 +c m_2^2)^s}\nonumber\\
&&=\frac{1}{4}u^{-(s-1/2)} \xi_1(2 s-1)+\frac{1}{4}u^{(s-1/2)} \xi_1(2 s)\nonumber \\
& &+\sum_{m_1,m_2=1,\infty}\left(\frac{m_1}{m_2}\right)^{(s-1/2)}
K_{s-1/2} (2\pi u m_1 m_2) \cos (2 \pi v m_1 m_2),
\label{kobereq}
\end{eqnarray}
where $v=b/a$ and $u^2+v^2=c/a$. The double sum on the left-hand side runs over all pairs of positive or negative integers, excluding the origin. This sum has been well studied, with the aim of discovering for what special values of $a$, $b$ and $c$ it can be represented as a finite superposition of products of Dirichlet $L$ functions- see Table 1.6 of Borwein {\em et al} \cite{lsbook} for known results. In general, if more than one Dirichlet $L$ product appears in the result for a sum, it will have non-trivial zeros off the critical line.

The Macdonald function sum on the right-hand side is rapidly convergent for the imaginary part of $s$ not large, but the region of summation required for accuracy increases linearly with the imaginary part.  This sum can have zeros off the critical line \cite{jmpus}.

The first term on the right-hand side has been studied by both Lagarias and Suzuki, and by Ki. Its zeros are governed
in the following result:
\begin{theorem}(Lagarias and Suzuki)
For each $y\ge 1$ the zeros of the function
\begin{equation}
a_0(y,s)=\xi_1(2 s) y^s+\xi_1(2-2 s) y^{1-s}
\label{lseq3}
\end{equation}
have the properties that, for $1\leq y\leq y^*=4\pi e^{-\gamma}\simeq7.055507$ all zeros lie on the critical line, while for $y>y^*$ there are exactly two off the critical line and in the unit interval.
\label{lsthm3}
\end{theorem}

Let $N(f,T)$ count the number of zeros of the function $f(s)$ having $0\leq t\leq T$.  Then Lagarias and Suzuki give the relationship
between zero distributions  of $a_0(y,s)$ and $\xi_1(2 s)$:
\begin{equation}
N(a_0(y,s),T)=N(\xi_1(2 s), T)+\frac{2}{\pi} \log (y) T+O(1),
\label{lseq4}
\end{equation}
where, from Titchmarsh and Heath-Brown \cite{titchmarsh}, Theorem 9.4,
\begin{equation}
N(\xi_1(2 s), T)=\frac{T}{\pi} \log T-\frac{T}{\pi} (\log\pi+1)+O(\log T).
\label{lseq5}
\end{equation}
Note that, from equation (\ref{lseq4}),
\begin{equation}
N(a_0(1,s),T)=N(\xi_1(2 s), T)+O(1).
\label{lseq4a}
\end{equation}

\subsection{The Results of Ki}
Ki proves results relating to the function
\begin{equation}
f(y,s)=(s-1)s(2 s-1) \xi_1(2 s) y^s+s(1-s)(1-2 s)  \xi_1(2-2 s) y^{1-s},
\label{ki1}
\end{equation}
which is odd under $y\rightarrow 1-s$. Then
\begin{theorem}(Ki) All complex zeros of $f(s)$ are simple and lie on the critical line.
\label{thmk1}
\end{theorem}
For $y=1$ this reduces to Taylor's result Theorem \ref{prtthm}.

Another valuable set of results relate to  the following function:
\begin{equation}
\theta (t)=\arg[(2it)(1/2+i t)(-1/2+ i t) \xi_1(1+2 i t)].
\label{ki2}
\end{equation}
Then:
\begin{lemma}(Ki) The function $\theta (t)$ has the properties: (1) $\theta(0)=\pi$, (2) $\theta(t)>\pi/2$ for $t>0$,
(3) $\theta(t)$ is a convex function in $(0,7)$ and (4) $\theta (t)$ increases in $[7,\infty)$, with $\theta '(t)\sim \log t$.  
\end{lemma}

\subsection{The Functions ${\cal T}_+(s)$ and ${\cal T}_-(s)$}

We continue with two discussions related to functions for which it is known that all non-trivial zeros are first-order and located
on the critical line \cite{ki,lagandsuz,mcp13}.

The first  function is defined as:
\begin{equation}
{\cal T}_+(s)=\frac{1}{4} [\xi_1(2 s)+\xi_1(2 s-1)].
\label{T+0}
\end{equation}
It has a pole of order unity at $s=0$:
\begin{equation}
{\cal T}_+(s)\sim-\frac{1}{8 s}+\frac{1}{24}(3\gamma+\pi-3\log(4\pi))+O(s).
\label{T+1}
\end{equation}
It tends to a constant at $s=1/2$:
\begin{equation}
{\cal T}_+(s)\sim \frac{1}{4} (\gamma-\log (4\pi))+O((s-1/2)^2).
\label{T+2}
\end{equation}
It has a pole of order unity at $s=1$:
\begin{equation}
{\cal T}_+(s)\sim \frac{1}{8 (s-1)}+\frac{1}{24}(3\gamma+\pi-3\log(4\pi))+O(s-1).
\label{T+3}
\end{equation}
It is even under $s\rightarrow 1-s$.

The function ${\cal T}_+(s)$ takes the following form on the critical line:
\begin{equation}
{\cal T}_+(1/2+i t)=2 |\xi_1(1+2 i t)| \cos[\arg(\xi_1(1+2 i t)|],
\label{T+CL}
\end{equation}
and thus its zeros correspond to $\arg(\xi_1(1+2 i t)=(n+1/2)\pi$ for  any integer $n$.

The author has compiled a list of the first 1517 zeros of ${\cal T}_+(s)$, the last of which is at $t=\Re(s)\simeq 999.912$. (A list of the first 1517 zeros of ${\cal T}_+(s)$ is available via the electronic supplementary material to this paper.)

The second  function is defined as:
\begin{equation}
{\cal T}_-(s)=\frac{1}{4} [\xi_1(2 s)-\xi_1(2 s-1)].
\label{T-0}
\end{equation}
This function is odd under $s\rightarrow 1-s$. It has poles at $s=0$, $s=1/2$ and $s=1$, and zeros on the real line at $s=3.91231$ and $s=-2.91231$. The function ${\cal T}_-(s)$ takes the following form on the critical line:
\begin{equation}
{\cal T}_-(1/2+i t)=2 i  |\xi_1(1+2 i t)| \sin[\arg(\xi_1(1+2 i t)|],
\label{T-CL}
\end{equation}
and thus its zeros correspond to $\arg(\xi_1(1+2 i t)=n \pi$ for  any integer $n$.

Once again, a list of the first 1517 zeros of ${\cal T}_-(s)$ is available via the electronic supplementary material to this paper.

Let
\begin{equation}
\theta_1 (t)=\arg[ \xi_1(1+2 i t)].
\label{thCL1}
\end{equation}
Then, from (\ref{ki2}), in $t>0$,
\begin{equation}
\theta (t)=-\frac{\pi}{2}+\theta_1 (t),
\label{thCL2}
\end{equation}
and
\begin{equation}
\theta' (t)=\theta_1' (t).
\label{thCL3}
\end{equation}

In so far as the properties of ${\cal T}_+(s)$ and ${\cal T}_-(s)$ are concerned, these are collected in the following Theorem:
\begin{theorem} The functions ${\cal T}_+(s)$ and ${\cal T}_-(s)$ have all their complex zeros on the critical line, and they occur in interlaced fashion, with a zero of the former lying between zeros of the latter, and vice versa. All zeros of each are simple, and the
distribution functions of each satisfy:
\begin{equation}
\label{dfnTpm}
N({\cal T}_+(s), T)=N({\cal T}_-(s), T)=N(\xi_1(2 s), T)=\frac{T}{\pi} \log T-\frac{T}{\pi} (\log\pi+1)+O(\log T).
\end{equation}
\label{Tpmprops}
\end{theorem}
\begin{proof}
These statements follow from the results of Taylor, Suzuki and Lagarias, and Ki, as well as from equation (\ref{thCL3}).
The first complex zero of ${\cal T}_+(s)$ occurs for $t\simeq 6.97468$, while that for ${\cal T}_-(s)$ occurs for $t\simeq 7.66111$, after which the interlacing property commences. We emphasise that in equation (\ref{dfnTpm}) the zeros counted all lie on the critical line for the first two functions, but for the third the zeros referred to lie in $0<\sigma<1/2$.
\end{proof}

It is interesting to investigate whether a positional relationship exists between the zeros of either  ${\cal T}_+(s)$  or ${\cal T}_-(s)$  and those of $\xi_1(2 s-1/2)$, which under the Riemann hypothesis all lie on the critical line.
We have investigated firstly whether zeros of $\xi_1(2 s-1/2)$ all lie after those of ${\cal T}_+(s)$. In fact, of the first 1500 zeros of the latter, this property does not hold in 232 or 15.5\% of  cases. We can also ask whether numbered  zeros of $\xi_1(2 s-1/2)$  lie between successive zeros of ${\cal T}_-(s)$. This property fails only in four cases: 921, 995, 1307 and 1495. Using translations we can make the property hold for all 1500 zeros: translations by $s_0=i t_0$ achieve this for $t_0$ in the range -0.080 to -0.036. This same translation along the critical line leads to a variation along it of the shifted ratio ${\cal T}_+(s)/{\cal T}_-(s)$ of the form
\begin{equation}
\frac{{\cal T}_+(1/2+i (t-t_0))}{{\cal T}_-(1/2+i (t-t_0))}=-i \cot [\arg \xi_1(1+2 i (t-t_0))].
\label{trans7}
\end{equation}
Hence, each point on the critical line can be made a zero of either the numerator or the denominator function in (\ref{trans7}).  In addition,
each zero of $\xi_1(2 s-1/2)$ on the critical line can be made to coincide with a zero of ${\cal T}_-(s-s_0)$ or of ${\cal T}_+(s-s_0)$, or to lie between two such.
\subsection{Properties of  ${\cal V} (s)$ and ${\cal W}(s)$}
From the functions ${\cal T}_+(s)$ and ${\cal T}_-(s)$ we construct two further functions:
\begin{equation}
{\cal V}(s)=\frac{{\cal T}_+(s)}{{\cal T}_-(s)}=\frac{1+{\cal U}(s)}{1-{\cal U}(s)},
\label{eq-srh1}
\end{equation}

Note that
\begin{equation}
{\cal V}(s)=\frac{1+\sqrt{\pi} \Gamma (s-1/2)\zeta (2 s-1)/(\Gamma (s)\zeta (2 s))}{1-\sqrt{\pi} \Gamma (s-1/2)\zeta (2 s-1)/(\Gamma (s)\zeta (2 s))} ,
\label{eq-srh3}
\end{equation}
leading to the first-order estimate for $|{\cal V}(s)|$ for $1<<|\sigma|<<t$
\begin{equation}
|{\cal V}(s)|\sim 1+\sqrt{\frac{2}{t}}.
\label{eq-srh4}
\end{equation}
The corresponding argument estimate in  $1<<\sigma<<t$ is
\begin{equation}
\arg [{\cal V}(s)] \sim -\sqrt{\frac{2 \pi}{t}}.
\label{eq-srh5}
\end{equation}
For $  {\cal U}(s)$, we have
 \begin{equation}
  {\cal U}(s)=
  \sqrt{\frac{\pi}{s}} (1+\frac{3}{8 s}+\ldots) (1+\frac{1}{4^s}+\frac{2}{9^s}+\ldots).
 \label{udef3}
 \end{equation}
 From this, if $t>>1$ and $\sigma>2$ say, $|{\cal U}(s)|\sim \sqrt{\frac{2}{t}}$ and $\arg {\cal U}(s)$ lies in the fourth quadrant.
 The leading terms in the expansion for ${\cal V}(s)$ are
 \begin{equation}
  {\cal V}(s) \sim 1+\sqrt{\frac{2}{t}}(1-i \sqrt{\pi})+O(1/t),
 \label{asymV}
 \end{equation}
 which places its  argument in the fourth quadrant, close to the real axis. For ${\cal W}(s)$ we have
 \begin{equation}
  {\cal W}(s) \sim -i+\sqrt{\frac{2}{t}}(1-i \sqrt{\pi})+O(1/t),
 \label{asymW}
 \end{equation}
 which places its  argument in the fourth quadrant, close to the boundary with the third quadrant .

We know that ${\cal V}(s)$ has all its non-trivial zeros and poles interlaced along the critical line, while the Riemann hypothesis is that
${\cal U}(s)$ has its zeros on $\Re (s)=3/4$ and its poles on $\Re (s)=1/4$. From (\ref{eq-srh1}), zeros of ${\cal U}(s)$ correspond 
to ${\cal V}(s)=1$ and poles to ${\cal V}(s)=-1$. Both then must lie on contours  of constant modulus $|{\cal V}(s)|=1$, which correspond to ${\cal U}(s)$ being pure imaginary: $\Re [{\cal U}(s)]=0$.

An investigation has been carried out into the relationship between the contours of constant modulus $|{\cal V}(s)|=1$ and the location of the zeros and poles of ${\cal U}(s)$, for the first 1500 zeros of ${\cal T}_+(s)$ and ${\cal T}_-(s)$. A convenient way of doing this in the symbolic/numerical/graphical package Mathematica is to use the option RegionPlot, and in this case to construct the regions in which $|{\cal V}(s)|\leq 1$,  or equivalently in which $\Re [{\cal U}(s)]<0$. These can be combined with contour plots of  $|{\cal V}(s)|$,
with contours appropriately chosen to highlight the location and behaviour around zeros of the derivative function ${\cal U}'(s)$, evaluated by numerical differentiation.

The results of this (rather labour intensive) investigation are quite suggestive. In each of the 1500 cases, a zero of ${\cal T}_+(s)$
on the critical line sits at the centre of a simply-connected region, whose boundary fully encloses the region $|{\cal V}(s)|\leq 1$.
The region $|{\cal V}(s)|> 1$ is multiply connected, in keeping with the estimate in equation (\ref{eq-srh3}) for $\sigma$ not too close to the critical line. Two examples are given in Figs. \ref{fig-srh1} and  \ref{fig-srh2}. The first example (zero 518 of ${\cal T}_+(s)$)  shows what may be described as  typical behaviour, while the second (zero 1495)  corresponds to one of the four exceptions mentioned in the previous section.

 \begin{figure}[tbh]
\includegraphics[width=7cm]{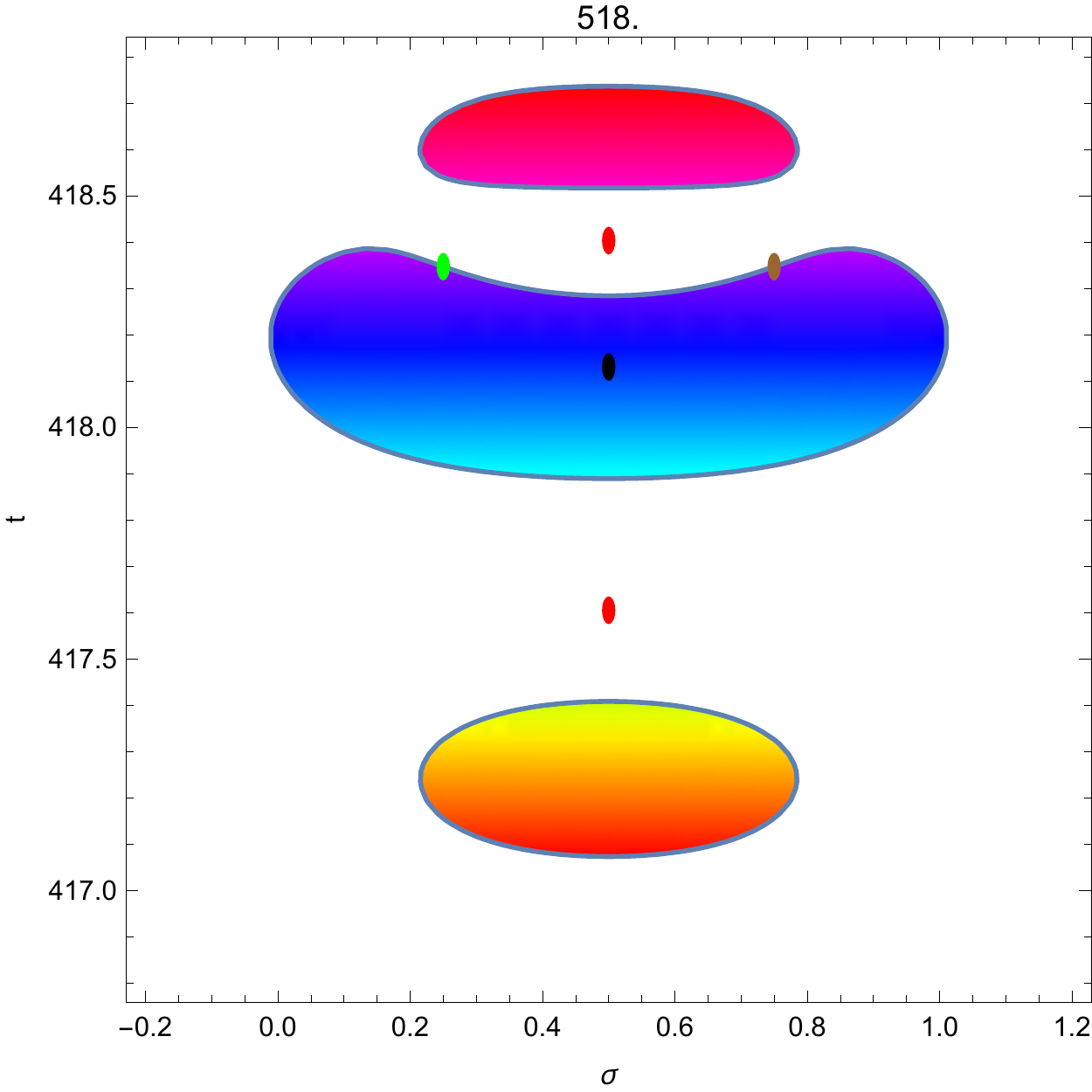}~\includegraphics[width=7cm]{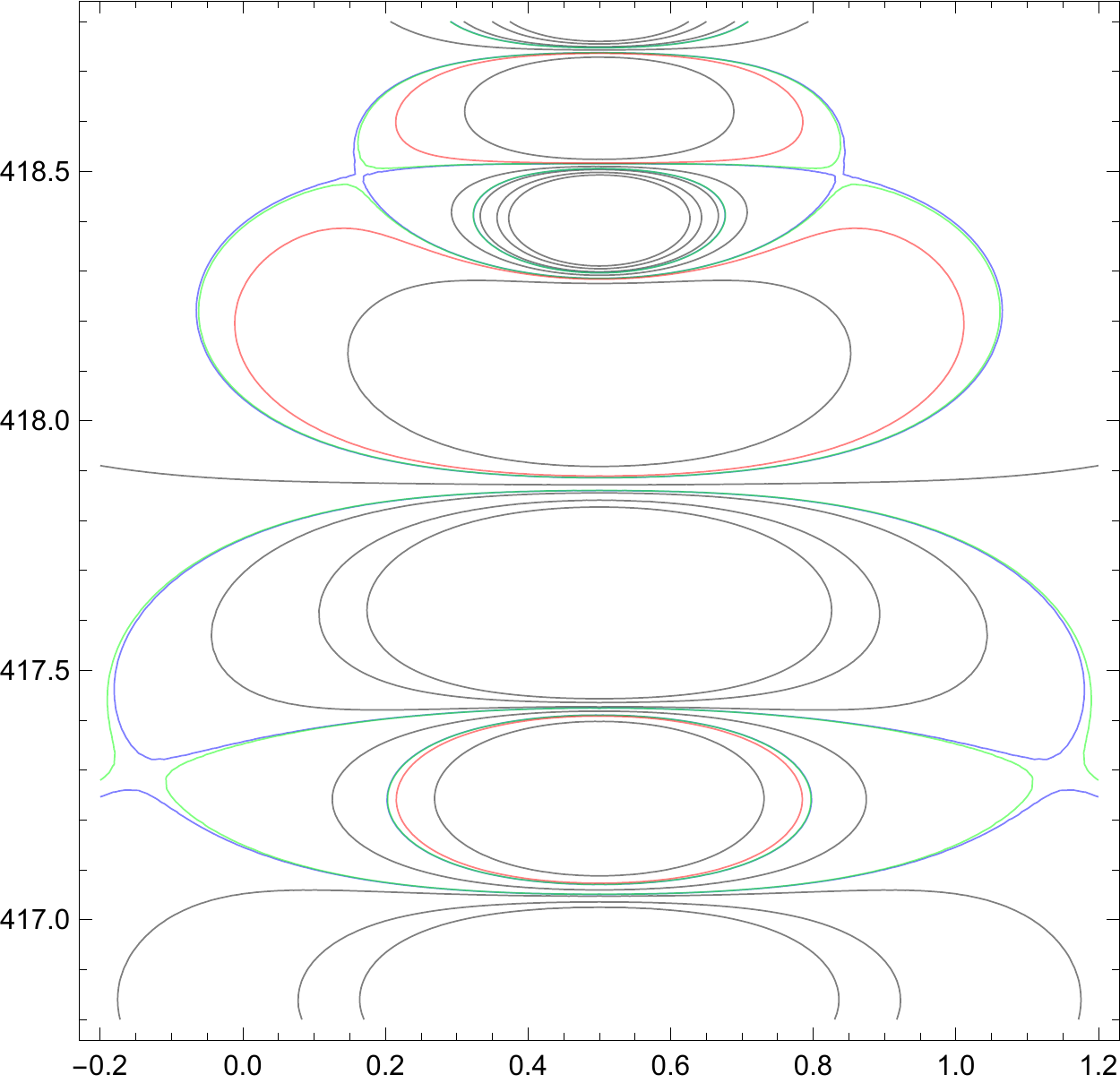}
\caption{At left, three successive regions (coloured)  in which $|{\cal V}(\sigma+i t)|\leq 1$ are shown. Black dots represent zeros, red dots poles, while green and brown dots denote zeros of $\xi_1( 2s)$ and $\xi_1(2 s-1)$. At right, corresponding contours of constant modulus, with red denoting unit modulus, while green and blue contours correspond to moduli just above and below the values corresponding to the derivative zeros of ${\cal V}(\sigma+i t)$. }
\label{fig-srh1}
\end{figure}

 \begin{figure}[tbh]
\includegraphics[width=7cm]{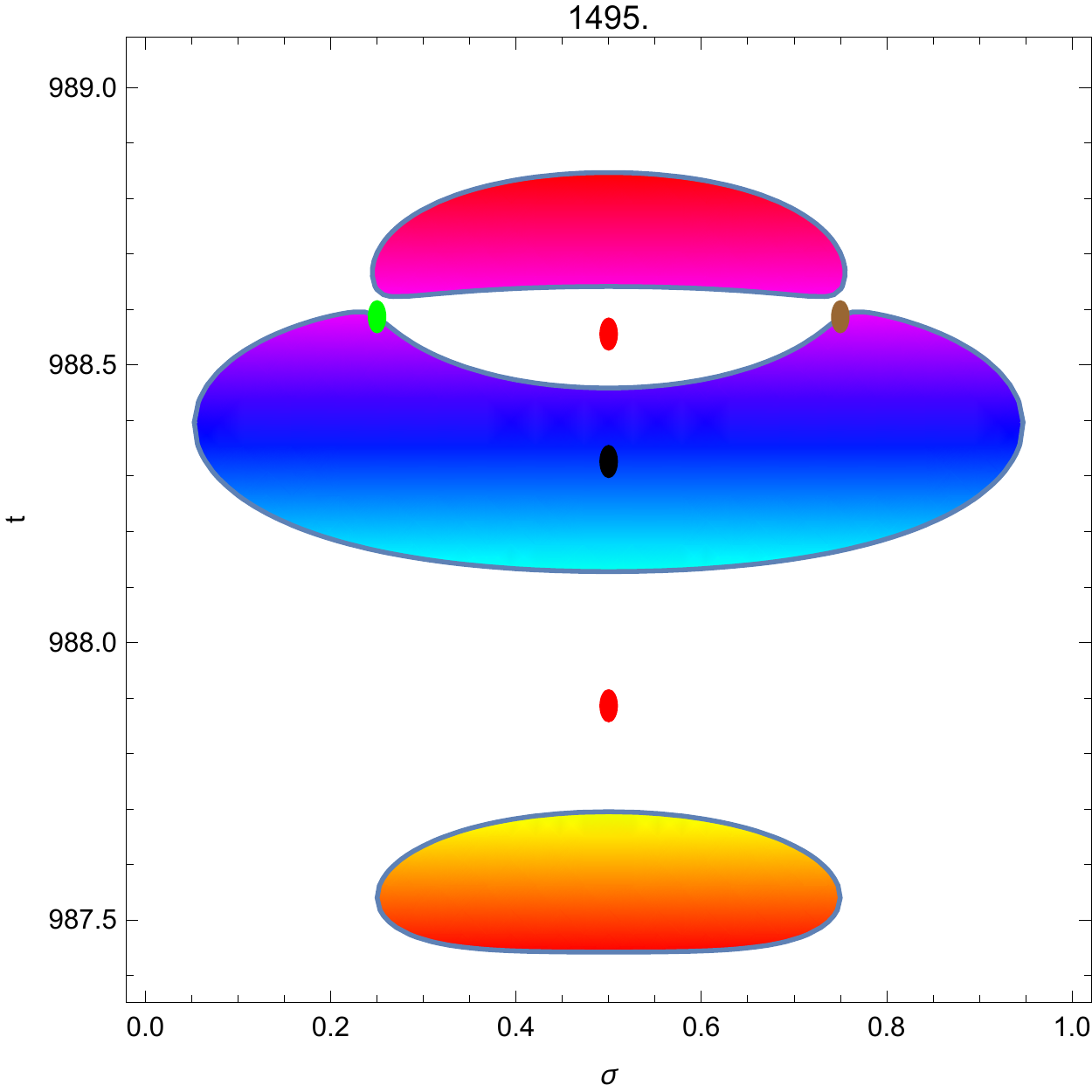}~\includegraphics[width=7cm]{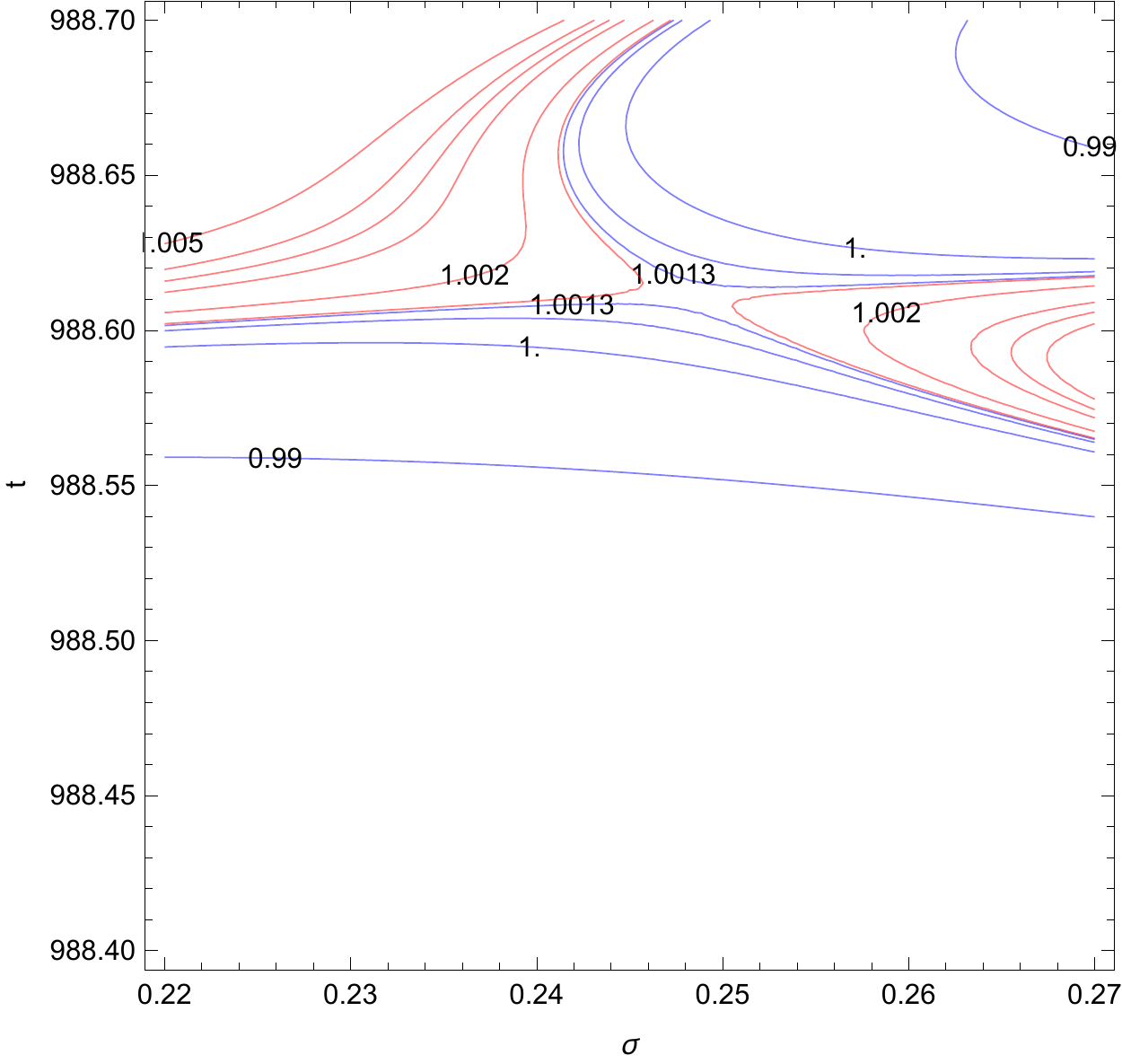}
\includegraphics[width=7cm]{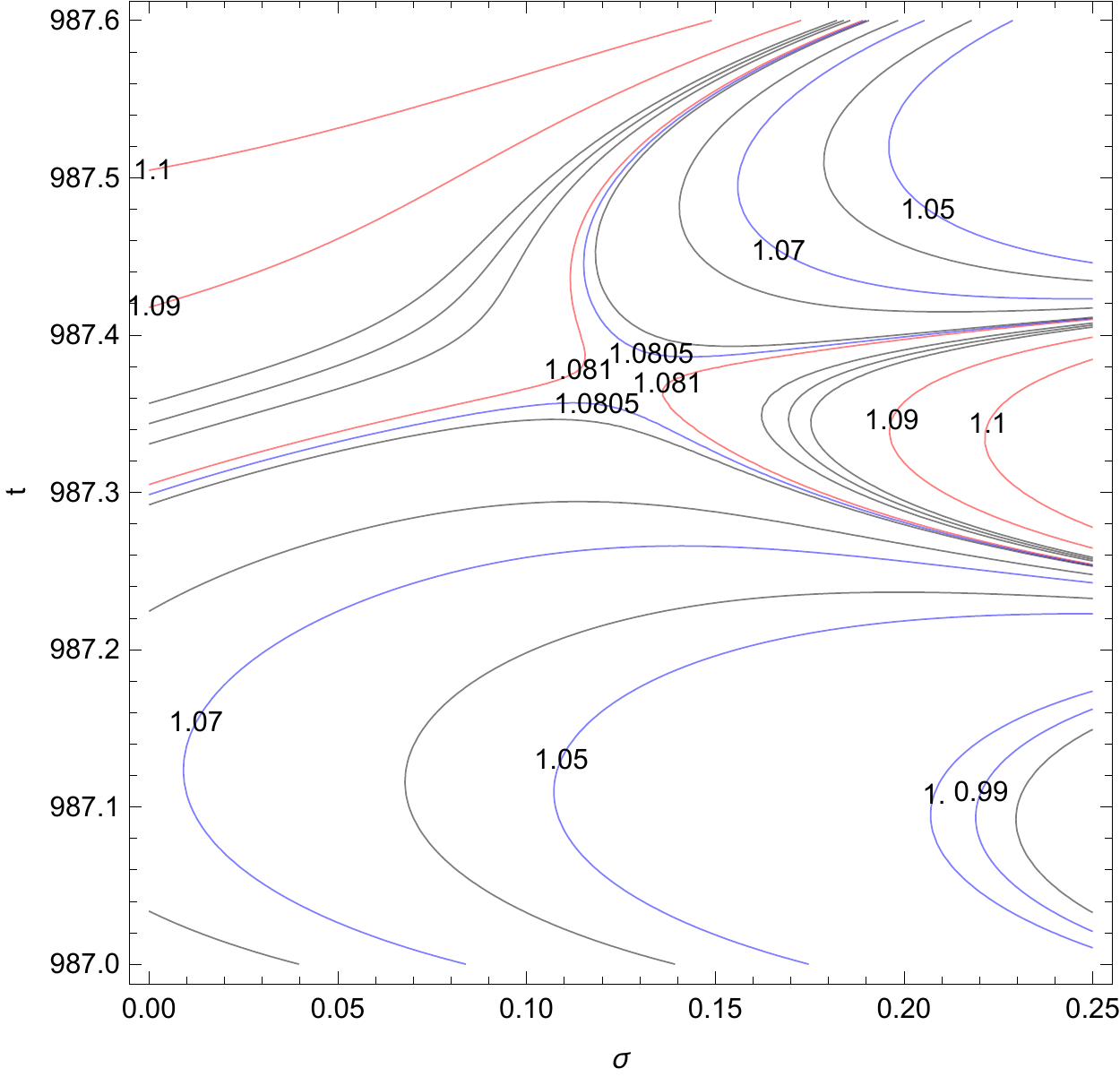}
\caption{At left, three successive regions (coloured)  in which $|{\cal V}(\sigma+i t)|\leq 1$ are shown. Black dots represent zeros, red dots poles, while green and brown dots denote zeros of $\xi_1( 2s)$ and $\xi_1(2 s-1)$. At right, corresponding contours of constant modulus, showing the region around one of the  derivative zeros of ${\cal V}(\sigma+i t)$. Below: detail of contours of constant modulus in the region of the lower derivative zero.}
\label{fig-srh2}
\end{figure}

In the first example, the contours of constant modulus shown are for the levels 0.90, 1.0, 1.018, 1.019, 1.1, 1.169, 1.170, 1.2, 1.3, 1.4.
The zeros of ${\cal U}'(s)$, or equivalently of ${\cal V}'(s)$, are approximately  $s=-0.143103+417.293 i$, where $|{\cal V}(s)|=1.16957$,
and $s=0.163301+ 418.4092 i$, where $|{\cal V}(s)|=1.01891$. The upper derivative zero is defined by four contours of constant modulus, two in green provided by the zeros of ${\cal V}(s)$, and two in blue, one pertaining to the intervening pole of ${\cal V}(s)$ and the other to a closed curve enclosing the two poles and one zero. The lower structure is not complete as shown, but the outermost curve encloses two poles and an intervening zero.

In the second example, the zeros of ${\cal U}'(s)$ are approximately
$s=0.24809+988.611 i$, where $|{\cal V}(s)|=1.001357$, and $s=0.12566+ 987.373 i$, where $|{\cal V}(s)|=1.0808$. The structure near the upper derivative zero could be described as nearly closed, with the two contours of constant modulus coming close to touching.
In consequence, the modulus at the derivative zero is much closer to unity than for the far more open structure around the lower derivative zero. For this example, both derivative zeros referred to are provided by two zeros of ${\cal V}(s)$ surrounding an intervening pole. 

A nice geometrical insight into the results of this section comes from writing (\ref{eq-srh1}) in the normal form \cite{knopp} already given in Section 2: see (\ref{wdef}).
This equation is built around the fixed points where ${\cal V}(s)=i={\cal U}(s)$ and ${\cal V}(s)=-i={\cal U}(s)$, which become the zeros and poles of ${\cal W}(s)$. Note that ${\cal W}(s)$ also incorporates the zeros and poles of ${\cal U}(s)$ and ${\cal V}(s)$:\newline
the poles and zeros of ${\cal V}(s)$ are where ${\cal W}(s)=1$ and ${\cal W}(s)=-1$ respectively;\newline
the poles and zeros of ${\cal U}(s)$ are where ${\cal W}(s)=i$ and ${\cal W}(s)=-i$ respectively.

Some further properties of ${\cal W}(s)$ are quoted next, taken from \cite{mcp13}, starting with its functional equation:
\begin{equation}
{\cal W}(1-s)=\frac{1}{{\cal W}(s)}.
\label{f1-1}
\end{equation}
It is real on the critical line, has unit modulus on the real axis, has the special values
${\cal W}(0)=-i$, ${\cal W}(1/2)=-1$, and in $t>>\sigma>>1$ has the asymptotic expansion
\begin{equation}
{\cal W}(\sigma+i t)=-i +(1-i)\sqrt{\frac{2 \pi}{t}}+\frac{2 \pi}{t}+(1+i)(\pi+\frac{\sigma}{2})\sqrt{\frac{2 \pi}{t^3}}+\ldots .
\label{f1-2}
\end{equation}
\label{thmf1}
All zeros and poles of ${\cal W}(s)$ lie on the critical line, and are interlaced:
\begin{equation}
 {\cal W}(1/2+i t)=\tan[\arg(\xi_1(1+2 i t))+\frac{\pi}{4}].
 \label{addedW}
 \end{equation}

Some properties of ${\cal W}(s)$ are exemplified in Fig. \ref{Wplot}, and Table \ref{table1}. Lines $|{\cal W}(s)|=1$ correspond to $\Im {\cal V}(s)=0$, which also correspond to $\Im {\cal U}(s)=0$. The lines of constant argument join at the zeros and poles of ${\cal W}(s)$, and cut $\sigma=1/2$ at zeros and poles of ${\cal V}(s)$, while cutting $\sigma=1/4, 3/4$ at zeros and poles of ${\cal U}(s)$. Points where
 ${\cal W}'(s)=0$ are also points where  ${\cal U}'(s)=0$ and  ${\cal V}'(s)=0$. In the example at right in Fig. \ref{Wplot}, the derivative zero is located close to $\sigma=1/4$, but still to its left. Note also that in this case the concavity of the constant modulus curves means  that the alternating property of zeros along these curves does not carry over to their $t$ values.
\begin{figure}
\includegraphics[width=7cm]{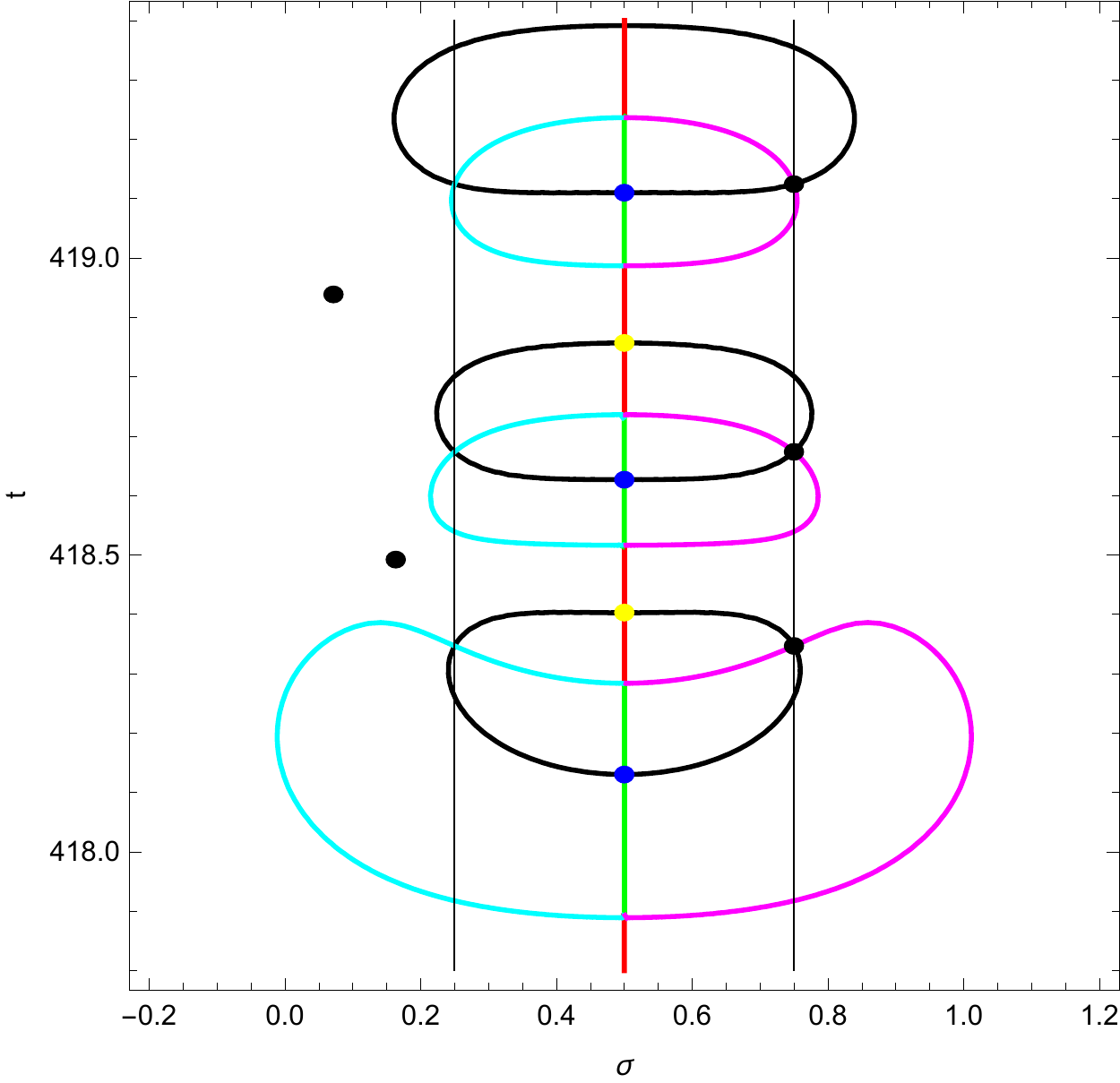}~~\includegraphics[width=7cm]{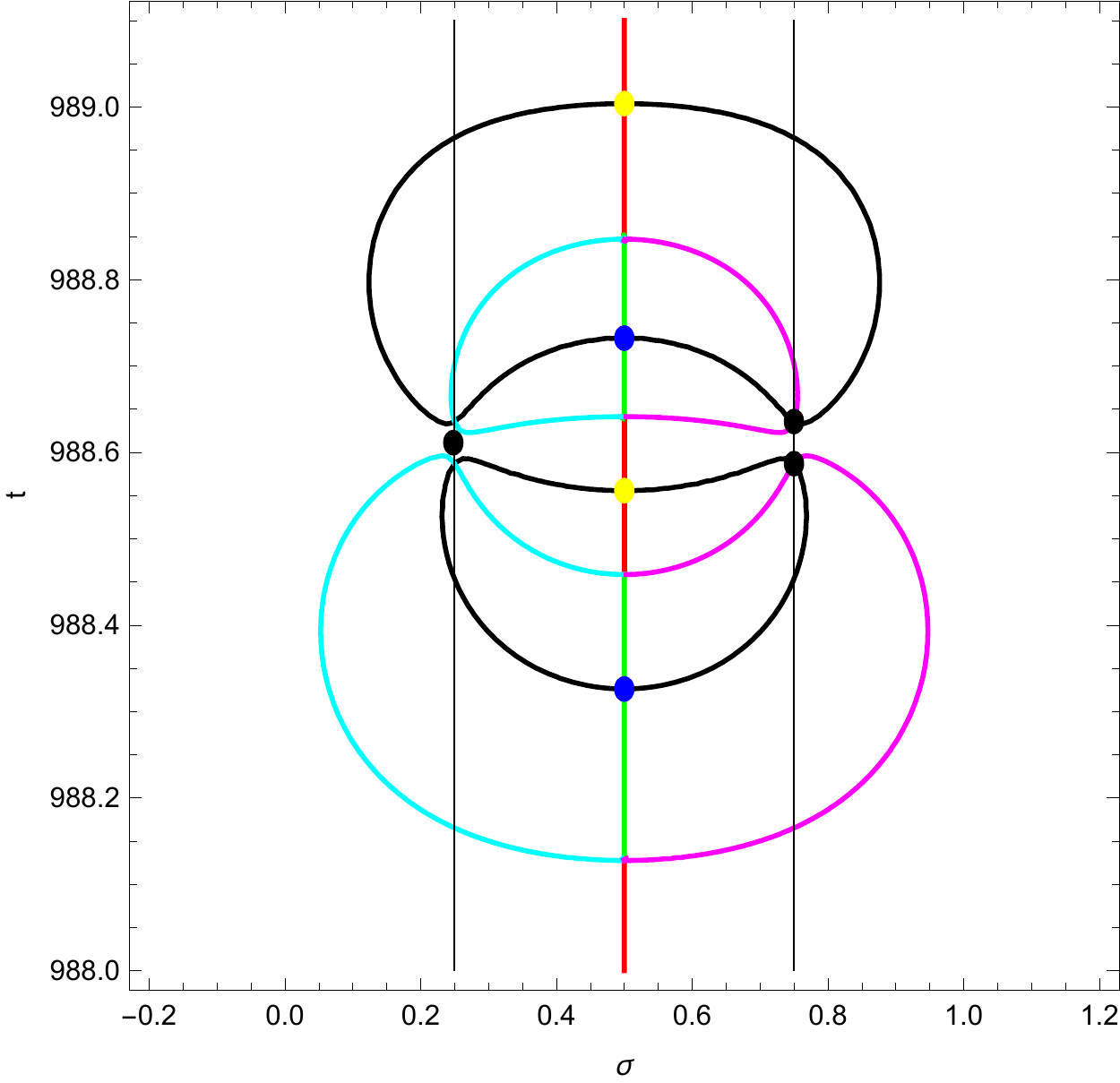}
\caption{  Contour plots showing (black lines) contours $|{\cal W}(s)|=1$. Coloured lines denote lines of constant  argument of 
${\cal W} (\sigma+ i t)$ as a function of $\sigma$  and $t$. The argument values are: $0$- red; $\pm \pi$- green, $\pi/2$-cyan; $-\pi/2$- magenta.
The dots on $\sigma=1/2$ are zeros (blue) and poles (yellow) of ${\cal V}(s)$; those on $\sigma =3/4$ are zeros of ${\cal U}(s)$, with the corresponding poles being located symmetrically on $\sigma=1/4$. The other  black dots show points where ${\cal W}'(s)=0$.}
\label{Wplot}
\end{figure}

It is easy to show that the equation $|{\cal V}(s)|=|{\cal W}(s)|$  has solutions if and only if  ${\cal U}(s)=\Re ({\cal U}(s)) (1-i)$, which places
$\arg {\cal U}(s)$ either in the second or fourth quadrants. The particular case where the common modulus is unity  arises if and only if
$\Re {\cal U}(s)=0=\Im {\cal U}(s)$.

\begin{table}
\begin{tabular}{|c|c|c|}\\ \hline
Quadrant & Line & Quadrant \\ \hline
Q2 & $|{\cal V}(s)|=1$ & Q1 \\
$|{\cal V}(s)|<1$, & $\arg {\cal U}(s)=\frac{\pi}{2}$ & $|{\cal V}(s)|>1$,\\ 
$|{\cal W}(s)|<1$ & & $|{\cal W}(s)|<1$  \\ \hline
$\arg {\cal U}(s)=\pm \pi$ & & $\arg {\cal U}(s)=0$ \\ \hline 
Q3 & $|{\cal V}(s)|=1$ & Q4 \\
$|{\cal V}(s)|<1$, & $\arg {\cal U}(s)=-\frac{\pi}{2}$ & $|{\cal V}(s)|>1$,\\ 
$|{\cal W}(s)|>1$ & & $|{\cal W}(s)|>1$  \\ \hline
Boundaries & Properties & \\ \hline
Q1-Q4: & $|{\cal W}(s)|=1$ & $|{\cal V}(s)|>1$ \\
Q1-Q2: & $|{\cal W}(s)|<1$ & $|{\cal V}(s)|=1$ \\
Q2-Q3: & $|{\cal W}(s)|=1$ & $|{\cal V}(s)|<1$ \\
Q3-Q4: & $|{\cal W}(s)|>1$ & $|{\cal V}(s)|=1$ \\ \hline
Diagonals & Properties & \\ \hline
Q1 & $\arg {\cal U}(s)=\pi/4$  & $|{\cal V}(s)|=1/|{\cal W}(s)|$\\
Q2 & $\arg {\cal U}(s)=3 \pi/4$ &  $|{\cal V}(s)|=|{\cal W}(s)|$\\
Q3 & $\arg {\cal U}(s)=-3\pi/4$ &  $|{\cal V}(s)|=1/|{\cal W}(s)|$\\
Q4 & $\arg {\cal U}(s)=- \pi/4$ &  $|{\cal V}(s)|=|{\cal W}(s)|$\\ \hline
\end{tabular}
\caption{Properties of ${\cal V}(s)$ and ${\cal W}(s)$ in the four quadrants of $\arg {\cal U}(s)$, on the lines separating them and on diagonals.}
\label{table1}
\end{table}

Given the triplet of functions ${\cal U}(s)$,  ${\cal V}(s)$ and  ${\cal W}(s)$, we can also construct relations other than (\ref{wdef}) in normal form. That involving ${\cal U}(s)$ and  ${\cal W}(s)$ is
\begin{equation}
\frac{{\cal W}(s)-1}{{\cal W}(s)+1}=i \left( \frac{{\cal U}(s)-1}{{\cal U}(s)+1}\right).
\label{uwnorm}
\end{equation}
That involving ${\cal V}(s)$ and  ${\cal W}(s)$ is more complicated.
Let
\begin{equation}
{\cal V}_1=\sqrt{2 + \sqrt{3}} \exp [-i \pi/4], ~{\cal V}_2=\sqrt{2 - \sqrt{3}} \exp [i 3 \pi/4],~ {\cal V}_3=-\frac{1}{2}+\frac{\sqrt{3}}{2} i.
\label{wvnorm1}
\end{equation}
Then the ${\cal W}$, ${\cal V}$ equation in normal form is
\begin{equation}
\frac{{\cal W}(s)-{\cal V}_1}{{\cal W}(s)-{\cal V}_2}={\cal V}_3 \left(\frac{{\cal V}(s)-{\cal V}_1}{{\cal V}(s)-{\cal V}_2}\right).
\label{wvnorm2}
\end{equation}

\section{Four Propositions}
%The arguments of Section 3 centred around the location of zeros of ${\cal U}(s)$ (or equivalently of ${\cal V}(s)$) and their relationship
%to the zeros and poles of these two functions. 
There exist in the extensive literature devoted to the Riemann hypothesis several relevant papers concerning the location of zeros of the derivative of the Riemann zeta function. An early example is the paper in German of
A. Speiser \cite{speiser}, for which detailed comments exist in English \cite{xray}. Spieser states that the Riemann hypothesis is equivalent to the proposition that the non-trivial zeros of $\zeta '(s)$ lie in $\sigma>1/2$, i.e. to the right of the critical line.  
Arias-de-Reyna comments that Speiser's methods
are "between the proved and the acceptable", and gives numerical examples illustrating Speiser's reasoning. A further comment is that a flawless proof of a stronger result is that due to Levinson and Montgomery \cite{levinson-montgomery}, which we  now quote.

\begin{theorem} (Levinson and Montgomery) 
The Riemann hypothesis is equivalent to $\zeta '(s)$ having no zeros in $0<\sigma<1/2$.
\end{theorem}

We now consider the  equivalent statement for ${\cal U}(s)$ for $1/2<\sigma<3/4$. As a stimulus for  this,we give in Fig. \ref{fig-srh3} graphs of $\arg {\cal U} (\sigma+ i t)$ for $\sigma=1/2$,  $\sigma=3/4$ and  $\sigma=0.753$. The notable point is that the argument decreases monotonically as $t$ increases for $\sigma=3/4$ (as it has been shown by Ki to do for $\sigma=1/2$), but for values in excess of $\sigma=3/4$ it is no longer monotonic, but contains segments in which it increases with $t$.

 \begin{figure}[tbh]
\includegraphics[width=7cm]{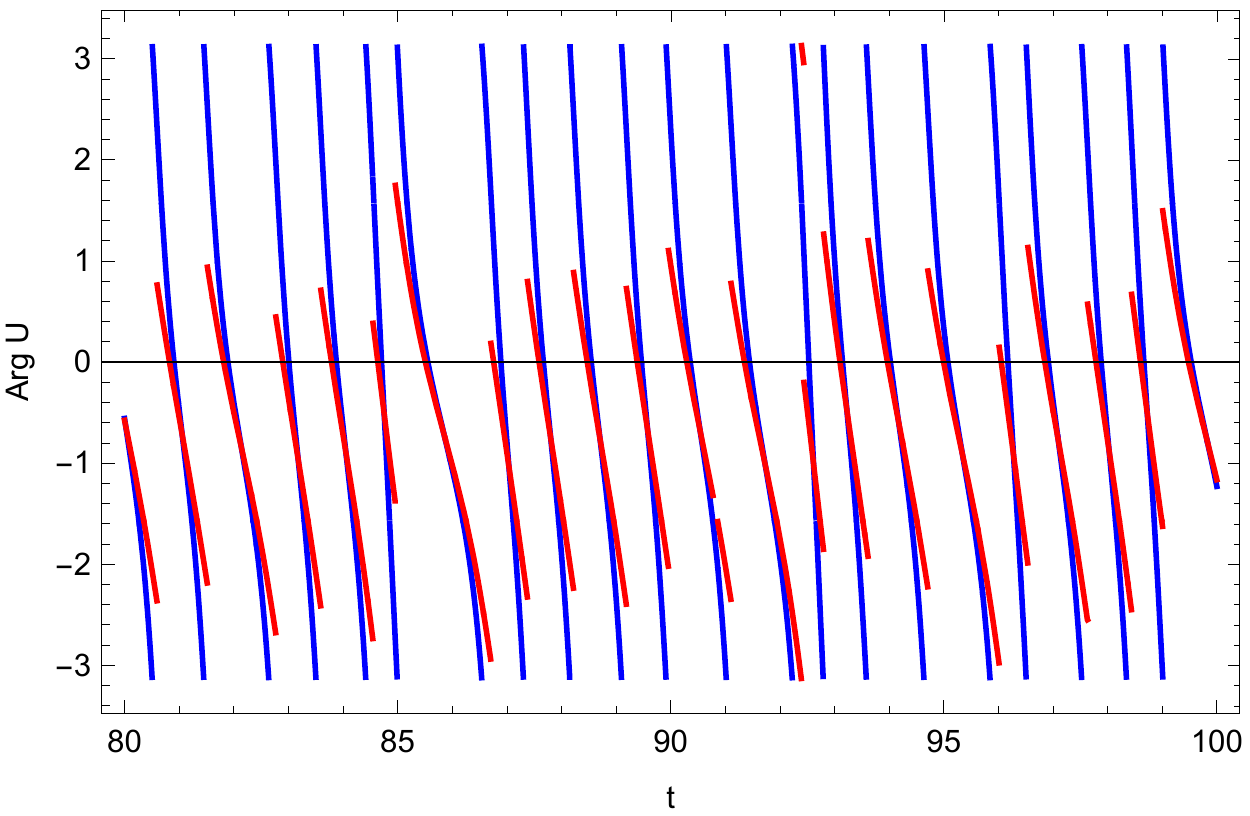}~\includegraphics[width=7cm]{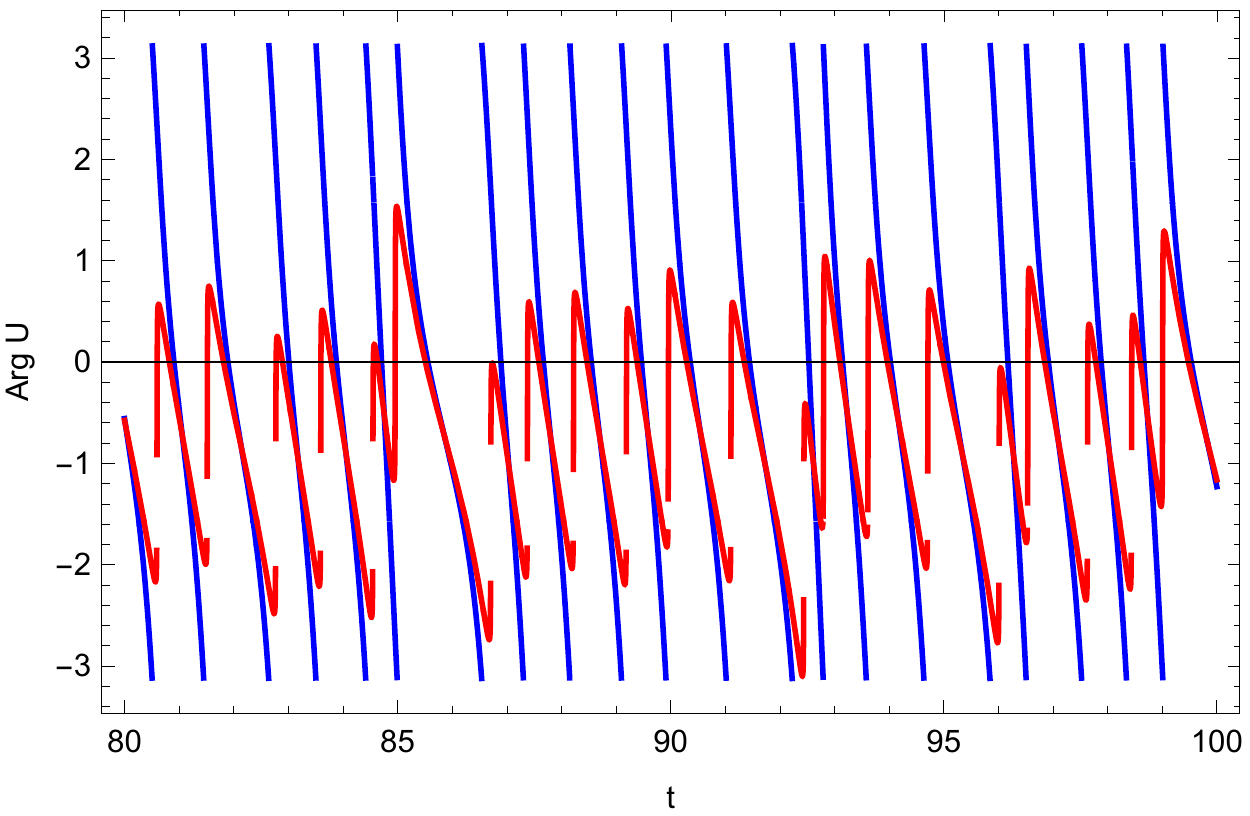}
\caption{ Plots of the argument of ${\cal U} (\sigma+ i t)$ for $\sigma=1/2$ (blue) and in red for (left) $\sigma=3/4$ and for (right) $\sigma=0.753$.}
\label{fig-srh3}
\end{figure}

We will discuss this question  in the context of four propositions, and discuss their relationships. In this section we will discuss the first two, and in the next the relationship between them and the second pair.

{\bf Proposition 1}: All non-trivial zeros of ${\cal U}(s)$ lie on $\sigma=3/4$, and all non-trivial poles on $\sigma=1/4$.

{\bf Proposition 2}: No zeros of ${\cal U}'(s)$ or ${\cal V}'(s)$ or ${\cal W}'(s)$  lie between or on the lines $\sigma=3/4$ and $\sigma=1/4$.

{\bf Proposition 3}: All  non-trivial zeros of ${\cal U}'(s)$ or ${\cal V}'(s)$ or ${\cal W}'(s)$  correspond to a modulus $|{\cal V}(s)|>1$.

{\bf Proposition 4}: The values ${\cal V}(1/2+ i t)=i$ and ${\cal V}(1/2+ i t)=-i$ are attained by contours of constant $|{\cal V}(s)|$ increasing in modulus from zero. 

Proposition 1 is of course the Riemann hypothesis.
%\subsection{Comments based on assuming Proposition 3}
%Given Proposition 3 is assumed, we established in Section 3 that almost all non-trivial zeros of ${\cal U}(s)$ lie on $\sigma=3/4$, and 
%almost all non-trivial poles on $\sigma=1/4$. In relation to Proposition 3, we note that in almost all cases where ${\cal V}'(s_0)=0$, we must have that $\sigma_0$ lies outside the closed interval between  $\sigma=3/4$ and $\sigma=1/4$. Indeed, if there is an $s_0$
%with $\sigma_0$ inside the closed interval, we must be able to construct a closed contour $|{\cal V}(s)|=1$ lying inside  $1/4<\sigma<3/4$, and so find a zero of ${\cal U}(s)$ lying off $\sigma=3/4$. This can only be possible in sufficiently few cases to not disturb the distribution function for such zeros. Thus, Proposition 3 implies that Propositions 2 and 3 hold in almost all cases.

\subsection{Comments based on assuming  Proposition 1}
The following Theorem is an extension of a discussion in Section 4 of \cite{mcp13}, and in particular of its Corollary 4.2.
\begin{theorem}
If Proposition 1 holds, then $\Re \frac{\partial}{\partial s} \log {\cal U}(s)<0 $ in $\frac{1}{4} \leq \sigma \leq \frac{3}{4}$ for $t>4.08046$.
\label{thm2to3}
\end{theorem}
\begin{proof}
We let $s_p$  for $p$ integral run over all non-trivial poles of ${\cal U}(s)$ in $t>0$, which by assumption can be written $s_p=1/4+i t_p$. The non-trivial zeros
of ${\cal U}(s)$ are then located at $s_p+1/2=3/4+i t_p$. Taking into account the zero and pole of ${\cal U}(s)$ on the real axis, we have the representation 
\begin{eqnarray}
\Re \frac{\partial}{\partial s} \log {\cal U}(s)&=&-\Re \left[  \frac{1}{s-1}-\frac{1}{s}\right]
-\sum_{p=1}^\infty m_p \Re \left[ \frac{1}{s-s_p}-\frac{1}{s-1/2-s_p}\right] \nonumber\\
 & & -\sum_{p=1}^\infty m_p \Re \left[ \frac{1}{s-{\bar s}_p}-\frac{1}{s-1/2-{\bar s}_p}\right].
\label{thm23-1}
\end{eqnarray}
Here the $m_p$ take into account possible zeros and poles of multiplicities exceeding unity. The equation (\ref{thm23-1})  simplifies:
\begin{eqnarray}
\Re \frac{\partial}{\partial s} \log {\cal U}(s)&=&-\Re \left[  \frac{1}{s(s-1)}\right]
+\frac{1}{2}\sum_{p=1}^\infty m_p \Re \left[ \frac{1}{(s-s_p)(s-1/2-s_p)}\right] \nonumber\\
 & & +\frac{1}{2}\sum_{p=1}^\infty m_p \Re \left[ \frac{1}{(s-{\bar s}_p)(s-1/2-{\bar s}_p) }\right].
\label{thm23-2}
\end{eqnarray}
We now substitute for $s_p$ as above  and obtain, after identification of the real part,
\begin{eqnarray}
\Re \frac{\partial}{\partial s} \log {\cal U}(s)&=& \left[  \frac{\sigma (1-\sigma)+t^2}{(\sigma^2 +t^2)((1-\sigma)^2+t^2)}\right]\nonumber \\
& &+\frac{1}{2}\sum_{p=1}^\infty m_p  \left[ \frac{(\sigma -1/4)(\sigma -3/4)-(t-t_p)^2 }{((\sigma -1/4)^2+(t-t_p)^2)((\sigma -3/4)^2+(t-t_p)^2)}\right] \nonumber\\
 & & +\frac{1}{2}\sum_{p=1}^\infty m_p  \left[ \frac{(\sigma -1/4)(\sigma -3/4)-(t+t_p)^2 }{((\sigma -1/4)^2+(t+t_p)^2)((\sigma -3/4)^2+(t+t_p)^2)}\right]. \nonumber  \\
 & & 
\label{thm23-3}
\end{eqnarray}
The expression (\ref{thm23-3}) consists of two parts. The first is positive and converges as $1/t^2$ for $t$ large. Its expansion for $|t|$ large starts as $1/t^2$. The second has all of its terms  negative if $1/4<\sigma<3/4$, and consists of two sums which converge quadratically as respectively $|t-t_p|$ and $|t+t_p|$ increase. Away from $\sigma=1/4$ and  $\sigma=3/4$, we may expand the dominant elements of the second part to leading order as
\begin{equation}
-\frac{1}{2}\sum_{p=1}^\infty \left[\frac{1}{1/4+(t-t_p)^2}+\frac{1}{1/4+(t+t_p)^2}\right].
\label{spart}
\end{equation}
The second part  is then readily seen to exceed the first part in modulus  when $t$ is not small. This establishes the result in between the lines of zeros and poles.

Finally, consider the special case $\sigma=3/4$ (with a similar argument applying to $\sigma=1/4$) . Then we obtain
\begin{equation}
\Re \frac{\partial}{\partial s} \log {\cal U}(s)=\frac{3/16+t^2}{(9/16+t^2)(1/16+t^2)}-\frac{1}{2}
\sum_{p=1}^\infty \left[\frac{1}{1/4+(t-t_p)^2}+\frac{1}{1/4+(t+t_p)^2}\right].
\label{thm23-4}
\end{equation}
Numerically, the  term $p=1$ in the sum has  larger magnitude than the first term on the right-hand side when $t>4.08046$. 
\end{proof}
Thus, Proposition 1 implies Proposition 2, 
and leads to an  important result following immediately from Theorem \ref{thm2to3}:
\begin{corollary} If the Riemann hypothesis holds, then all zeros of $\zeta (s)$ on the critical line are simple.
\label{multzer}
\end{corollary}
\begin{proof}
All zeros of $\zeta (s)$ for $t$ up to $O(10^9)$ are known numerically to lie on the critical line, and to be simple. From the consideration of  $\Re \frac{\partial}{\partial s} \log {\cal U}(s)$ on the lines $\sigma=3/4, 1/4$,we know that  ${\cal U}'(s)$ has no zeros there.
Taking into account the properties  of the prefactor of $\zeta(2 s-1)/\zeta(2 s)$ in the expression (\ref{eq-srh2}) for ${\cal U}(s)$, this leads to the result.
\end{proof}
\subsection{The equivalence of Propositions 1 and 2}
To establish the equivalence of Propositions 1 and 2, we need the complementary statement: if Proposition 1 fails, then so does Proposition 2.  To  motivate the discussion of this, we  multiply the function ${\cal V}(s)$ by a function which introduces an extra pair of zeros at $s_{\pm}=3/4\pm \delta +i t_*$,
together with appropriate
balancing terms to  preserve symmetry under $s\rightarrow 1-s$ and complex conjugation:
\begin{equation}
{\cal F}(s,\delta, t_*)=\frac{[s-(3/4+\delta+i t_*)][s-(3/4+\delta-i t_*)][s-(3/4-\delta+i t_*)][s-(3/4-\delta-i t_*)]}
{[s-(1/4+\delta+i t_*)][s-(1/4+\delta-i t_*)][s-(1/4-\delta+i t_*)][s-(1/4-\delta-i t_*)]}.
\label{equiv231}
\end{equation}
We replace $s$ by $\sigma+ i t$, and define $\tilde{\sigma}=\sigma-3/4$, $\hat{\sigma}=\sigma-1/4 =\tilde{\sigma}-1/2$, $t_{\pm}=t \pm t_*$. We then obtain:
\begin{equation}
{\cal F}(\sigma, t,\delta, t_*)=\frac{[(\tilde{\sigma}+i t_-)^2-\delta^2][(\tilde{\sigma}+i t_+)^2-\delta^2]}
{[(\hat{\sigma}+i t_-)^2-\delta^2][(\hat{\sigma}+i t_+)^2-\delta^2]}.
\label{equiv232}
\end{equation}
We then define functions incorporating the zeros and poles off $\sigma=3/4$ and $\sigma=1/4$:
\begin{equation}
{\cal U}_{oa}(s,\delta, t_*)={\cal U}(s) {\cal F}(s,\delta, t_*),
\label{equiv233}
\end{equation}
and
\begin{equation}
{\cal V}_{oa}(s,\delta, t_*)=\frac{1+{\cal U}_{oa}(s,\delta, t_*)}{1-{\cal U}_{oa}(s,\delta, t_*)},
\label{equiv233v}
\end{equation}
\begin{equation}
{\cal W}_{oa}(s,\delta, t_*)=\frac{{\cal V}_{oa}(s,\delta, t_*)-i}{ {\cal V}_{oa}(s,\delta, t_*)+i}.
\label{equiv233w}
\end{equation}

The artificial example of this construction is given in Fig. \ref{offaxisWV}, for the case $\delta=0.05$, $t_*=418.85$. There are closed
contours of unit modulus as shown before in  Fig. \ref{Wplot}, together with  new closed contours giving the inner zero and pole of
${\cal U}_{oa}(s,0.05,418.85)$, and open contours corresponding to its outer pole and zero.  However, of most interest to this argument
is the position of the derivative zero of  ${\cal U}_{oa}$ at the point $0.737209 + 418.847 i$ (together with another at  $0.262791 + 418.847 i$). This indicates a breakdown of Proposition 3 due to that of Proposition 2 in this artificial case.
          
\begin{figure}
\includegraphics[width=7cm]{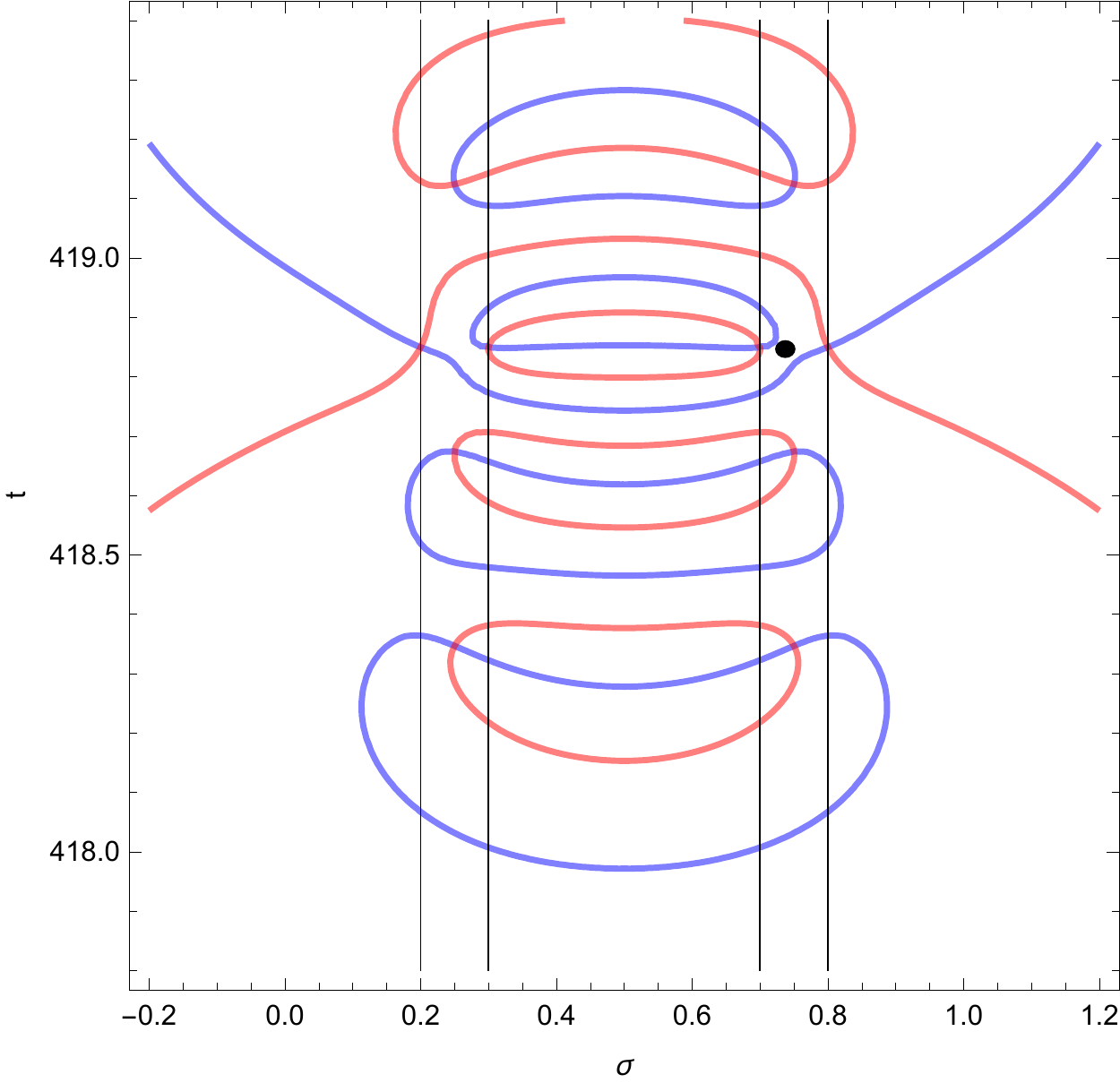}
\caption{A   contour plot showing (blue lines) contours $|{\cal W}_{oa}(s,0.05,0.418.85)|=1$ and (red lines) contours $|{\cal V}_{oa}(s,0.05,418.85)|=1$. Artificial poles and zeros have been introduced on the black lines $\sigma=0.20,0.30, 0.70,0.80$, at positions where the red and blue contours intersect. Other zeros and poles are as in Fig. \ref{Wplot} (left).The black dot marks the position of a new derivative zero. Note that the new closed contour $|{\cal V}_{oa}(s,0.05,418.85)|=1$ is centred on a pole on the critical line near $t=418.853$.}
\label{offaxisWV}
\end{figure}

\begin{theorem}
If ${\cal U}(s)$ has zeros off $\sigma=3/4$ then ${\cal U}'(s)$ has zeros between $\sigma=1/4$ and $\sigma=3/4$.
\label{notp23}
\end{theorem}
\begin{proof}
We split the product representation of ${\cal U}(s)$ into a part with zeros and poles on $\sigma=3/4,1/4$ respectively, and a part which is a product of terms of the form (\ref{equiv232}), and with zeros and poles off $\sigma=3/4,1/4$:
\begin{equation}
{\cal U}(s,\delta, t_*)={\cal U}_{on}(s) {\cal U}_{off} (s)={\cal U}_{on}(s) \prod_q{\cal F}(s,\delta_q, t_q).
\label{equiv233}
\end{equation}
We can apply Theorem \ref{thm2to3} to ${\cal U}_{on}(s)$, which is thus monotonic decreasing in modulus for non-small $t$ as $\sigma$ increases between $1/4$ and $3/4$.

As far as $ {\cal U}_{off} (s)$ is concerned, let us take $q=1$ to be the case we are interested in investigating, where certainly $t_1^*$ is positive and very much larger than unity. Considering the term $q=1$ in isolation first, for $t$ close to $t_1$:
\begin{equation}
{\cal F}(\sigma, t,\delta_1, t_1)=\frac{[(\tilde{\sigma}+i t_-)^2-\delta_1^2][(1-i\tilde{\sigma}/(2  t_{1}))^2-\delta^2/(2 t_1)^2]}
{[(\hat{\sigma}+i t_-)^2-\delta_1^2][(1-i\hat{\sigma}/(2 t_1))^2)-\delta^2/(2 t_1)^2)]}.
\label{not231}
\end{equation}
This is zero when $\tilde{\sigma}=\pm 1/2$. Certainly $\delta_1<1/4$, so that the ratio of the second terms in the numerator and denominator in (\ref{not231}) can be replaced by unity when we look for the derivative zero on the line $t_-=0$ i.e. $t=t_1$. This derivative zero
occurs when
\begin{equation}
2 \tilde{\sigma} ( \tilde{\sigma} ^2+1/4+ \tilde{\sigma} -\delta_1^2)=( \tilde{\sigma} ^2-\delta_1^2)(2 \tilde{\sigma} +1),
\label{not232}
\end{equation}
or
\begin{equation}
 \tilde{\sigma} =-\frac{1}{4}(1\pm\sqrt{1-16 \delta_1^2}\sim -2\delta_1^2.
\label{not233}
\end{equation}
It is natural that this should be $O(\delta_1^2)$ since ${\cal F}$ is invariant under $\delta_1\rightarrow-\delta_1$. The derivative zero corresponds to a maximum modulus of this term.

The effect of the term $ {\cal U}_{off} (s) $ is thus to move the derivative zero into the region $1/4<\sigma<3/4$. Since $|{\cal U}_{on}(s)|$ is monotonic decreasing for $\sigma$ increasing in this region, the effect of this term is to move the derivative zero of the modulus of the product
further into the region, as well as to move it laterally off the line $t=t_1$ in general.

We must finally consider the effect of the possible terms $q\ne 1$ on the position of the derivative zero. Let $u_{q1}=t_q-t_1$. The effect of the term from $q$ is then a multiplicative factor
\begin{equation}
1+\frac{i/u_{q1}}{1-i/u_{q1}} \left(1+ \frac{2(t-t_1-i\tilde{\sigma})/u_{q1}}{1-i/u_{q1}}\right), 
\label{not234}
\end{equation}
which will not change the position of the derivative zero  to leading order (since the leading order derivative contribution is independent of $\sigma$).
\end{proof}

Note that, in the artificial example in Fig. \ref{offaxisWV}, the derivative zero of $|{\cal F}|$ occurs at $0.744949+ 415.85 i$, giving a shift
from the  line  $\sigma=3/4$ of $-0.005051$, agreeing well with $-2\delta_1^2$. The derivative zero of $| {\cal U}_{on} (s) {\cal U}_{off} (s)|$ occurs at
$0.737209+418.847 i$, so the shift from the line  $\sigma=3/4$ is $-0.01721$, with the movement  due to $|{\cal U}_{on}(s)|$ exceeding that due to $\delta_1$.
\section{Propositions 3 and 4}
We commence with a result establishing a link between the equivalent Propositions 1, 2 and Proposition  3.
\subsection{Comments based on  Propositions 1 and 2}
\begin{theorem}
Propositions 1 and 2 imply Proposition 3.
\label{P23g1}
\end{theorem}
\begin{proof}
Given Propositions 1 and 2 hold, we wish to exclude the possibility that Proposition 3 does not hold. If this were the case, we would have  an $s_0$ existing in $t_0>>1$ such that $|{\cal V}(s_0)|<1$, where ${\cal V}'(s_0)=0$ . The contours of constant modulus of ${\cal V}(s)$ around its pertinent zero would then be limited to below $|{\cal V}(s_0)|<1$, and thus would not reach 
the points ${\cal V}(s)=\pm i$ on the critical line. These then would have to be attained by a contour of constant modulus unity around the adjacent pole of ${\cal V}(s)$, which also passes through a zero of ${\cal U}(s)$ on the line $\sigma=3/4$. Further such contours then would have to exist  ranging down in modulus towards $|{\cal V}(s_0)|<1$, and cutting the line $\sigma=3/4$ with moduli smaller than unity before nearing $s_0$. This gives a contradiction, as $\sigma_0$ has to be above $\sigma=3/4$.
\end{proof}
\subsection{The equivalence of   Propositions 3 and 4}
{\bf Remark:} The only possible closed contours whose boundary is an equimodular contour of ${\cal V}(s)$ in $t$ not small  cut the critical line twice.
This is a simple consequence of the Maximum/Minimum Modulus Theorems, since the only non-trivial zeros and poles of ${\cal V}(s)$ 
lie on the critical line. Furthermore, contours of constant modulus touching the critical line are precluded since ${\cal V}'(s)$ is never zero on the critical line.

\begin{theorem}
Proposition 3 holds if and only if  Proposition 4 holds.
\label{P3equiv4}
\end{theorem}
\begin{proof}
If Proposition 3 holds, then to each zero of ${\cal V}(s)$ on the critical line corresponds a set of contours of constant modulus, with that modulus able to increase until a zero of ${\cal V}'(s)$ is encountered. As that zero must occur at a point where $|{\cal V}(s)|>1$, then the set of contours must include the contour of modulus unity. This intersects the critical line at the points ${\cal V}(s)=\pm i$.

Conversely, if Proposition 4 holds, then lines of constant modulus constructed around each zero must attain unity, where on the critical line ${\cal V}(s)=\pm i$. As there are no zeros of ${\cal V}'(s)$ on the critical line, the set of lines of constant modulus must go beyond unity, showing that the relevant zero of ${\cal V}'(s)$ corresponds to a modulus in excess of unity.
\end{proof}

{\bf Remark:}  If either of Propositions 3 and 4 fail, then there must exist contours of constant modulus based around a pole of ${\cal V}(s)$ attaining the relevant zero of ${\cal V}'(s)$, at which the modulus of $|{\cal V}(s)|<1$.

Note that this situation occurs in the artificial example of Fig. \ref{offaxisWV}. However, by the Theorems of this section, it cannot occur
for ${\cal V}(s)$ in the range of $t$ for which the Riemann hypothesis is known to hold numerically.
\subsection{From Proposition 3 towards Proposition 1}
\begin{theorem}
If Proposition 3 holds, then there exists a set  of simple zeros of ${\cal U}(s)$ off the critical line in one-to-one correspondence with the zeros of ${\cal T}_+(s)$ on the critical line.
\label{suffthm}
\end{theorem}
\begin{proof}
The proof generalises the reasoning of Macdonald  \cite{macd} to functions having poles and zeros. We note that 
$|{\cal V}(s)|=|{\cal V}(1-\overline{s})|$ is a symmetric function under reflection in the critical line, so curves of constant modulus share this property. Starting from a general zero of ${\cal T}_+(s)$, we constrain the family of closed curves on which its modulus is constant.
This family of curves has as its final member that curve of constant modulus touching a zero of ${\cal V}'(s)$, and so by the assumption of this Theorem the corresponding constant modulus exceeds unity. It then follows that there is a closed curve of constant modulus unity enclosing the zero of ${\cal T}_+(s)$. This closed curve intersects the critical line at points where ${\cal V}(s)={\cal U}(s)=\pm i$,
and at every point on it ${\cal U}(s)$ is pure imaginary. The curve encloses one simple zero of ${\cal V}(s)$, and thus the argument of this function increases monotonically through a range of $2\pi$ along it, ensuring that it passes through $\pm 1$ upon it. On each such curve of constant modulus unity  there is then a simple pole and a simple zero of ${\cal U}(s)$, establishing the one-to-one correspondence referred to in the Theorem statement.
\end{proof}

\begin{corollary}
If Proposition 3 holds, then the distribution function for zeros of $\xi_1(2 s-1/2)$ on the critical line with $0<t<T$ is given by (\ref{dfnTpm}).
\label{distbfncor}
\end{corollary}
\begin{proof}
From Theorem \ref{suffthm}, there exists for each zero of ${\cal T}_+(s)$ on the critical line a simple zero of ${\cal U}(s)$(i.e. of $\xi_1(2 s-1)$) lying on the contour $|{\cal T}_+(s)/{\cal T}_-(s)|$ enclosing the zero of ${\cal T}_+(s)$. By Theorem \ref{Tpmprops} the distribution function for zeros of ${\cal T}_+(s)$ is the same as that for the  number of zeros of $\xi_1( 2 s-1/2)$ lying in the strip $1/4<\sigma<3/4$ with $0<t<T$. If any such zero were to lie off the critical line, it would have to occur in a pair of zeros symmetric about the critical line. However, we have established that there exists the set of simple zeros of $\xi_1(2 s)$ in precise 1:1 correspondence with the zeros of  ${\cal T}_+(1/2+ i t)$. Hence, any zeros off the critical line would have to be sufficiently rare to leave the distribution function (\ref{dfnTpm}) unaltered. In other words, the distributions functions for zeros {\it on the critical line} of $\xi_1( 2 s-1/2)$, ${\cal T}_+(s)$  and ${\cal T}_-(s)$ are all the same and given by (\ref{dfnTpm}),
if Theorem \ref{suffthm} holds. 
\end{proof}
\subsection{Proposition 3 and Topology}
\begin{theorem}
Assuming Proposition 3, to every closed contour $|{\cal V}(s)|=1$ around a zero of ${\cal V}(s)$ corresponds a  closed contour $|{\cal W}(s)|=1$ around a zero of ${\cal W}(s)$.
\label{Wtheorem}
\end{theorem}
\begin{proof}
Given a closed contour $|{\cal V}(s)|=1$ around a point where ${\cal U}(s)=-1$. From Table 1, this consists of elements in quadrants
$Q_2$ and $Q_3$ of $\arg {\cal U}(s)$ joined by a line $|{\cal W}(s)|=1$ where $\arg {\cal U}(s)=\pm \pi$. On the closed contour $|{\cal V}(s)|=1$ lie a zero and a pole of ${\cal U}(s)$. All four quadrants of $\arg {\cal U}(s)$ meet at the zero and pole of ${\cal U}(s)$. The 
area with ${\cal U}(s)$ in quadrant $Q_2$ is also a region where $|{\cal W}(s)|<1$. Outside the contour $|{\cal V}(s)|=1$ we have regions in $Q_4$ and $Q_1$, the former also being a region with $|{\cal W}(s)|>1$. The $Q_1$ region has $|{\cal W}(s)|<1$ and is joined to the $Q_2$ region along a contour $|{\cal V}(s)|=1$, $\arg {\cal U}(s)=\pi/2$. The desired region $|{\cal W}(s)|<1$ is then formed by the union of the $Q_1$, $Q_2$ regions.
\end{proof}
\begin{corollary} 
Given Proposition 3, then points where ${\cal W}'(s)=0={\cal U}'(s)={\cal V}'(s)$ have $\arg {\cal U}(s)$ lying in the fourth quadrant.
\label{corollW}
\end{corollary}
\begin{proof}
This follows directly from Theorem \ref{Wtheorem} and Table 1, given that regions $|{\cal V}(s)|\leq 1$ correspond to regions $|{\cal W}(s)|\leq 1$.
\end{proof}

{\bf Remark:} From the arguments of this subsection, it is a consequence that to every zero of ${\cal T}_+(s)$ corresponds a region
around it formed by the union of regions with $\arg {\cal U}(s)$ in the first, second and third quadrants, enclosed within a region with $\arg {\cal U}(s)$ in the fourth quadrant. It should be noted that the  fourth quadrant region is simply connected, while the other three quadrants are found in island regions within the fourth quadrant region.

\section{Supplementary Data}
Data on the zeros up to $t=1000$ on the critical line of the functions ${\cal T}_+(s)$ and ${\cal T}_+(s)$ is available in the electronic supplementary material. 

\section{Acknowledgements}
The author acknowledges gratefully contributions from colleagues (C.G. Poulton, L.C. Botten, the late N-A. Nicorovici) to previous published work
 which laid the foundations for what is reported here.


\begin{thebibliography}{9}  
\bibitem{edwards} Edwards HM, 2001 {\it Riemann's Zeta Function},  Mineola NY:Dover  pp. 132-134.
\bibitem{titchmarsh}Titchmarsh EC, Heath-Brown DR,  1986 {\it The Theory of the Riemann Zeta-function}, Oxford UK: Oxford .
\bibitem{borbail}Borwein, J. \& Bailey, D.  2004. {\it Mathematics by Experiment: Plausible Reasoning in the 21st Century},Natick, Mass: A.K. Peters. ISBN 978-1-56881-211-3.
\bibitem{bland} Bohr, H. \& Landau, E., 1914  Eine Satz \"{u}ber Dirichletsche Reihen mit Anwendung auf die $\zeta$-Funktion und die L-Funktionen Rend. di Palermo {\it 37} 269-72.
\bibitem{bcy} Bui, H.M., Conrey J.B. \& Young, M.P. 2011 More than 41\% of the zeros of the zeta
function are on the critical line Acta Arith. {\it150} 35-64.
\bibitem{prt} Taylor, P.R. 1945 On the Riemann zeta-function,  Q.J.O., {\it 16}, 1-21.
\bibitem{lagandsuz} Lagarias, J.C. and Suzuki, M. 2006 The Riemann hypothesis for certain integrals of Eisenstein series  J. Number Theory {\it 118} 98-122.
\bibitem{ki} Ki  H  2006 Zeros of the constant term in the Chowla-Selberg formula Acta Arithmetica {\it 124} 197-204.
\bibitem{lsbook} Borwein JM, Glasser ML, McPhedran RC, Wan JG \& Zucker IJ, 2013 {\it Lattice Sums Then and Now}, Cambridge UK: Cambridge.
%\bibitem{arxiv2018} McPhedran RC 2018 Sum Rules for Functions of the Riemann Zeta Type, arXiv:1801.07415.
\bibitem{speiser} Speiser, A. 1934  Geometrisches zur Riemannschen Zetafunktion, Math. Ann. {\it 110}  514╨
521.
\bibitem{chowsel} Chowla S \& Selberg A 1967 On Epstein's zeta-function  J. Reine Angeq. Math. {\it 227} 86-110.
\bibitem{kober}  Kober H  1935 Transformationsformeln gewisser Besselscher Reihen,
Beziehungen zu Zeta-Funktionen Math. Z.  {\it 39} 609-624. 
\bibitem{jmpus} McPhedran RC Smith GH Nicorovici NA \& Botten LC 2004  Distributive and analytic properties of lattice sums
 JMathPhys {\it 45}, 2560-2578.
 \bibitem{mcp13} McPhedran, R.C. \& Poulton, C.G. 2013 The Riemann Hypothesis for
Symmetrised Combinations of Zeta Functions, arXiv:1308.5756.
\bibitem{knopp} Knopp, K., 1952 {\it Elements of the Theory of Functions}, Dover NY Chapter 5.
\bibitem{xray}  Arias-de-Reyna, J. 2003 X-ray of Riemann zeta-function,arXiv:2003.09433.
\bibitem{levinson-montgomery} Levinson, N., \& Montgomery, H. L. 1974  Zeros of the Derivatives of the Riemann Zeta-Function,
Acta Math.{\it 133},  49╨65.
\bibitem{macd} McPhedran, R.C. 2017 Macdonald's Theorem for Analytic Functions, arXiv:1702.03458.

\end{thebibliography}
\end{document}